\documentclass[11pt]{amsart}

\usepackage{corona}
\usepackage{tikz}
\usetikzlibrary{matrix,arrows}

\newcommand{\proj}[2]{#2(#1)}
\newcommand{\quo}[1]{\pi\left(#1\right)}
\newcommand{\quomap}{\pi}
\newcommand{\homo}{\alpha}
\newcommand{\gomo}{\beta}
\newcommand{\ord}{\gamma}

\newcommand{\A}[1]{\mathcal{#1}}
\newcommand{\PA}[1]{\mathfrak{#1}}
\newcommand{\PC}[1]{\mathbf{#1}}
\newcommand{\SN}[1]{\MakeUppercase{#1}}
\newcommand{\SSN}[1]{\mathscr{#1}}
\newcommand{\T}[1]{\mathscr{#1}}

\title[Reduced products of UHF algebras]{Reduced products of UHF algebras under forcing axioms}
\author{Paul McKenney}

\begin{document}

\tikzset{cdg.smallmatrix/.style={matrix of math nodes, row sep=1em,
  column sep=0.8em, text height=1.5ex, text depth=0.25ex}}
\tikzset{cdg.matrix/.style={matrix of math nodes, row sep=3em,
  column sep=2.5em, text height=1.5ex, text depth=0.25ex}}
\tikzset{cdg.widematrix/.style={matrix of math nodes, row sep=3em,
  column sep=5em, text height=1.5ex, text depth=0.25ex}}
\tikzset{cdg.path/.style={->, font=\scriptsize}}
\tikzset{cut/.style={fill=white, inner sep=2pt}}

\begin{abstract}

  Let $\A{A}_n$ and $\A{B}_n$ ($n\in\NN$) be separable, unital UHF algebras.  We prove, assuming Todor\v cevi\'c's Axiom and Martin's Axiom, that every isomorphism of the form
  \[
    \prod \A{A}_n / \bigoplus \A{A}_n \simeq \prod \A{B}_n / \bigoplus \A{B}_n
  \]
  is \emph{definable} in a strong sense.  This confirms a conjecture of Coskey and Farah for this class of corona algebras.  We show moreover that the restriction of such an isomorphism to a C*-subalgebra of the form
  \[
    \prod \A{F}_n / \bigoplus \A{F}_n
  \]
  where each $\A{F}_n\subseteq\A{A}_n$ is finite-dimensional, must be determined by a sequence of $*$-homomorphisms $\A{F}_n\to \A{B}_{e(n)}$, where $e : \NN\to\NN$ is a function independent of the choice of the sequence $\A{F}_n$.   We prove as a corollary that $\prod\A{A}_n / \bigoplus \A{A}_n \simeq \prod\A{B}_n / \bigoplus \A{B}_n$ if and only if there is a function $e : \NN\to\NN$ such that for all but finitely many $n\in\NN$, $\A{A}_n \simeq \A{B}_{e(n)}$.

\end{abstract}

\maketitle

\section{Introduction}
\label{sec:intro}

Let $\A{B}$ be a C*-algebra and $\A{I}$ a (norm-closed, two-sided, $*$-closed) ideal in $\A{B}$; $\A{I}$ is called an \emph{essential} ideal in $\A{B}$ if for any other ideal $\A{J}$ in $\A{B}$, $\A{I}\cap\A{J} = 0$ implies $\A{J} = 0$.  Given a C*-algebra $\A{A}$, the \emph{multiplier algebra} $\A{M}(\A{A})$ is defined, up to isomorphism, to be the maximal unital C*-algebra containing $\A{A}$ as an essential ideal.  The \emph{corona algebra} of $\A{A}$ is the quotient $\A{Q}(\A{A}) = \A{M}(\A{A}) / \A{A}$. 

Corona algebras take their name from the following special case.  Let $X$ be a locally compact Hausdorff space and consider $C_0(X)$, the C*-algebra of continuous functions $X\to\CC$ which vanish at infinity.  Then $\A{M}(C_0(X))$ is isomorphic to the C*-algebra $C(\beta X)$, and $\A{Q}(C_0(X))$ is isomorphic to $C(\beta X\sm X)$, where $\beta X$ denotes the \v Cech-Stone compactification of $X$.  Spaces of the form $\beta X\sm X$ are themselves called \emph{corona spaces} and make up a central object of study in set-theoretic topology.  Corona algebras are abundant in the noncommutative setting as well; for instance, consider the C*-algebra $\A{K}(H)$ of compact operators on a Hilbert space $H$.  $\A{M}(\A{K}(H))$ is isomorphic to $\B(H)$, the C*-algebra of bounded operators on $H$, and $\A{Q}(\A{K}(H))$ is thus the \emph{Calkin algebra} over $H$, $\B(H)/\K(H)$.  Yet another example is given by the quotient $\prod \A{A}_n / \bigoplus \A{A}_n$, where each $\A{A}_n$ ($n\in\NN$) is a unital C*-algebra.  Here $\prod \A{A}_n$ denotes the C*-algebra of norm-bounded sequences , and $\bigoplus \A{A}_n$ the C*-algebra of sequences converging to zero.

Corona algebras are important in the theory of C*-algebras due to their connections with a wide array of topics, including projectivity and semiprojectivity of C*-algebras, stability of relations on C*-algebra generators, and the theory of extensions of C*-algebras.  (See for instance~\cite{Loring.LSPP}, \cite{Blackadar.KT}, and~\cite{Blackadar.OA}.)  They also have interesting behavior under certain set-theoretic hypotheses.  Extensive study has been given in particular to their automorphism groups, under the assumption of the Continuum Hypothesis ($\CH$), and, alternately, the Proper Forcing Axiom ($\PFA$).  (See~\cite{Rudin},~\cite{Shelah-Steprans.PFAA},~\cite{Farah.AQ},~\cite{Velickovic.OCAA},~\cite{Farah-McKenney.ZD} for the commutative case; and~\cite{Phillips-Weaver},~\cite{Farah.CO},~\cite{Farah.AC}, and~\cite{Coskey-Farah} for the noncommutative case.)  Typically, $\CH$ implies that there are many automorphisms, whereas $\PFA$ implies that the only automorphisms are those present in any model of $\ZFC$.  For example, $\CH$ implies that there are $2^{2^{\aleph_0}}$-many automorphisms of both $C(\beta\NN\sm\NN)$ and the Calkin algebra.  (See~\cite{Rudin} and~\cite{Phillips-Weaver}, respectively).  On the other hand, $\PFA$ implies that every automorphism of $C(\beta\NN\sm\NN)$ is induced by a function $e : \NN\to\NN$, and every automorphism of the Calkin algebra is inner.  In~\cite{Coskey-Farah}, Coskey and Farah considered the automorphisms of a general corona algebra, and found a notion of triviality which, in the cases of $C(\beta\NN\sm\NN)$ and the Calkin algebra $\A{Q}(\A{K}(H))$, turns out to hold exactly for those automorphisms described above.  Before we state their definition, recall that the \emph{strict topology} on a multiplier algebra $\A{M}(\A{A})$ is the weakest topology making the following seminorms continuous;
\[
  m\mapsto \norm{ma} + \norm{am} \qquad (m\in \A{M}(\A{A}),\, a\in\A{A})
\]
\begin{defn}
A $*$-homomorphism $\vp : \A{Q}(\A{A})\to\A{Q}(\A{B})$ is called \emph{trivial} if its \emph{graph},
\[
  \Gr{\vp} = \set{(a,b)\in \A{M}(\A{A})_1 \times \A{M}(\A{B})_1}{\vp(a + \A{A}) = b + \A{B}}
\]
is a Borel subset of $\A{M}(\A{A})_1\times \A{M}(\A{B})_1$ when each factor is endowed with the strict topology.
\end{defn}
(We emphasize that the graph of a $*$-homomorphism $\vp : \A{Q}(\A{A})\to\A{Q}(\A{B})$ is not the graph of a function in the usual sense, but the result of pulling this set back through the quotient maps $\A{M}(\A{A})\to\A{Q}(\A{A})$ and $\A{M}(\A{B})\to\A{Q}(\A{B})$.)  $\A{M}(\A{A})_1$ here refers to the unit ball of $\A{M}(\A{A})$.  The following conjectures made in~\cite{Coskey-Farah} extend all currently known results on automorphisms of corona algebras;
\begin{conj}
  \label{conj.CH}
  The Continuum Hypothesis implies that the corona of any separable, nonunital C*-algebra must have a nontrivial automorphism.
\end{conj}
\begin{conj}
  \label{conj.FA}
  Forcing axioms imply that every automorphism of the corona of a separable, nonunital C*-algebra must be trivial.
\end{conj}

(See also~\cite{Farah.AQ},~\cite{Farah.RC},~\cite{Just.EU} and~\cite{Farah-Shelah.TA} for work on the analogous conjectures for quotients of $\SSN{P}(\NN)$ by analytic P-ideals.)  In~\cite{Coskey-Farah}, Coskey and Farah prove Conjecture~\ref{conj.CH} for a large class of C*-algebras, including simple and stable C*-algebras.  In this paper, however, we will mainly be concerned with Conjecture~\ref{conj.FA}.  In place of the vague term ``forcing axioms'' we will use two combinatorial consequences of the Proper Forcing Axiom, \emph{Todor\v cevi\'c's Axiom} and \emph{Martin's Axiom}, which we will abbreviate as $\TA$ and $\MA$ respectively.  $\TA$ is also well-known as the \emph{Open Coloring Axiom}, $\OCA$, and can be viewed as a Ramsey-theoretic dichotomy for graphs on a set of real numbers; $\MA$ is the prototypical forcing axiom for ccc posets.  These principles have no large-cardinal strength relative to $\ZFC$ and can be forced over any model of set theory.  The reader is referred to~\cite{Todorcevic.PPIT},~\cite{Moore.PFA} for more information on their use in set theory and other areas of mathematics.

The main result of this paper is a confirmation of Conjecture~\ref{conj.FA} for a certain class of corona algebras, assuming $\TA + \MA$.  Before the statement we again need a definition.  A (separable, unital) \emph{UHF algebra} is a C*-algebra which can be realized as a direct limit of full matrix algebras over $\CC$, with unital connecting maps.

\begin{thm}
  \label{main.borel}
  Assume $\TA + \MA$, and let $\A{A}_n$ and $\A{B}_n$ ($n\in\NN$) be sequences of separable, unital UHF algebras.  If $\vp$ is an isomorphism of the form
  \[
    \prod \A{A}_n / \bigoplus \A{A}_n \simeq \prod \A{B}_n / \bigoplus \A{B}_n
  \]
  then $\Gamma_\vp$ is Borel.
\end{thm}

In proving Theorem~\ref{main.borel} it will be necessary to consider some stronger forms of triviality for $*$-homomorphisms between corona algebras.  We say that a map $\homo : \A{M}(\A{A})\to \A{M}(\A{B})$ is a \emph{lift} of $\vp : \A{Q}(\A{A})\to\A{Q}(\A{B})$ if the following diagram commutes;
\[
  \begin{tikzpicture}
    \matrix (m) [cdg.matrix] {
       \A{M}(\A{A}) &  \A{M}(\A{B}) \\
      \A{Q}(\A{A}) & \A{Q}(\A{B}) \\
    };
    \path [cdg.path]
      (m-1-1) edge node[auto]{$\homo$} (m-1-2)
      (m-2-1) edge node[auto]{$\vp$} (m-2-2)
      (m-1-1) edge (m-2-1)
      (m-1-2) edge (m-2-2);
  \end{tikzpicture}
\]
We call $\homo$ an \emph{algebraic lift} of $\vp$ if $\homo$ is a $*$-homomorphism, and \emph{strict} if $\homo$ is continuous with respect to the strict topologies.  Note that if $\vp$ has a strict lift then $\Gamma_\vp$ is necessarily Borel.  This conclusion does not always hold for $*$-homomorphisms with algebraic lifts, however; see Section~\ref{subsec:topologies} for further details.  In the process of proving Theorem~\ref{main.borel} we demonstrate the following;
\begin{thm}
  \label{main.cfh}
    Assume $\TA + \MA$.  Let $\A{A}_n$ and $\A{B}_n$ ($n\in\NN$) be separable, unital UHF algebras, and let $\vp$ be an isomorphism of the form
    \[
      \prod \A{A}_n / \bigoplus \A{A}_n \to \prod \A{B}_n / \bigoplus \A{B}_n
    \]
    Then for every sequence $\A{F}_n\subseteq\A{A}_n$ ($n\in\NN$) of finite-dimensional, unital C*-subalgebras, the restriction of $\vp$ to the C*-subalgebra $\prod\A{F}_n / \bigoplus \A{F}_n$ has a strict algebraic lift.
\end{thm}
Theorem~\ref{main.cfh} allows us to code a given isomorphism by what we call a \emph{coherent family of $*$-homomorphisms}.  Coherent families appear in various forms throughout the set-theoretic literature; see, for instance, \cite{Farah.AQ}, \cite{Farah.CO}, \cite{Todorcevic.PID}, and~\cite{Dow-Simon-Vaughan}.  The proof of Theorem~\ref{main.borel} is then completed with the following;
\begin{thm}
  \label{main.cfh->borel}
  Assume $\TA$.  Let $\A{A}_n$ ($n\in\NN$) be separable, unital UHF algebras, and let $\A{B}$ be any separable C*-algebra.  Suppose $\vp$ is a $*$-homomorphism of the form
  \[
    \prod \A{A}_n / \bigoplus \A{A}_n \to \A{M}(\A{B}) / \A{B}
  \]
  and there is a strict, algebraic lift of each restriction of $\vp$ to a unital C*-subalgebra of the form $\prod \A{F}_n / \bigoplus \A{F}_n$, with each $\A{F}_n \subseteq \A{A}_n$ finite-dimensional.  Then $\Gamma_\vp$ is Borel.
\end{thm}

It is interesting to note that Theorem~\ref{main.cfh} already gives us a form of rigidity for this class of corona algebras;
\begin{cor}
  \label{main.cor}
  Assume $\TA + \MA$.  Let $\A{A}_n$ and $\A{B}_n$ ($n\in\NN$) be separable, unital UHF algebras, and suppose
  \[
    \prod \A{A}_n / \bigoplus \A{A}_n \simeq \prod \A{B}_n / \bigoplus \A{B}_n
  \]
  Then there are cofinite sets $\SN{A},\SN{B}\subseteq\NN$ and a bijection $e : \SN{A}\to\SN{B}$ such that for all $n\in\SN{A}$, $\A{A}_n$ and $\A{B}_{e(n)}$ are isomorphic.
\end{cor}
Corollary~\ref{main.cor} essentially reduces the study of isomorphisms in our class of corona algebras to a study of automorphisms.  Unfortunately, we are not able to provide a strict, algebraic lift for a given automorphism of a corona algebra in our class, like we are able for its restrictions (assuming $\TA + \MA$).  The obstruction seems to be entirely C*-algebraic; in fact, we demonstrate in Section~\ref{sec:coherent-families} that the statement ``all trivial automorphisms of $\prod \A{A}_n / \bigoplus \A{A}_n$ have strict algebraic lifts'' is equivalent to an asymptotic form of an intertwining property.  (See~\cite{Elliott.intertwining} for more on Elliott's intertwining theorem.)  The proof goes through Theorem~\ref{main.cfh->borel} and an appeal to Schoenfield's absoluteness theorem.

The paper is structured as follows.  In Section~\ref{sec:preliminaries} we review some of the background needed for the proofs of the above results, including some standard tools from both combinatorial and descriptive set theory.  In particular we introduce the assumptions for Theorem~\ref{main.cfh}, $\TA$ and $\MA$.  We also provide a stratification of $\prod \A{A}_n / \bigoplus \A{A}_n$ into algebras of the form $\prod \A{F}_n / \bigoplus \A{F}_n$, where each $\A{F}_n$ is finite-dimensional; a similar stratification of the Calkin algebra underlies the proof of the main result of~\cite{Farah.CO} (see also~\cite[\S{4}]{Coskey-Farah}, and~\cite[Theorem~3.1]{Elliott.derivations-II}).  This allows us to introduce coherent families of $*$-homomorphisms; we then prove Corollary~\ref{main.cor}.  In Section~\ref{sec:definable-embeddings} we prove the following, in $\ZFC$.  Let $\vp$ be an injective, $*$-homomorphism of the form
\[
  \prod \A{F}_n / \bigoplus \A{F}_n\to \A{M}(\A{A}) / \A{A}
\]
where $\A{A}$ is an AF algebra (i.e. a direct limit of finite-dimensional C*-algebras), and each $\A{F}_n$ is finite-dimensional.  Suppose further that $\Gamma_\vp$ is Borel; then $\vp$ must in fact have a strict algebraic lift.  In Section~\ref{sec:fa-embeddings} we prove Theorem~\ref{main.cfh} by showing that the restriction of a given isomorphism to a subalgebra of the form $\prod\A{F}_n / \bigoplus \A{F}_n$ must have a Borel graph.  The work of this section is derived from arguments in~\cite{Velickovic.OCAA} and~\cite{Farah.AQ}.  In Section~\ref{sec:coherent-families} we prove Theorem~\ref{main.cfh->borel}, and discuss the question of whether all trivial automorphisms of a given corona algebra $\prod \A{A}_n / \bigoplus \A{A}_n$ have strict, algebraic lifts.

\section{Preliminaries}
\label{sec:preliminaries}

\subsection{Set theory}
\label{subsec:set theory}

We will assume that the reader is familiar with the basics of modern set theory as outlined in, for instance,~\cite{Kunen.2011}.  Our notation will for the most part follow the standards of the literature in set theory; in particular, we identify $\NN$ with the first infinite ordinal $\omega$, and $n\in\NN$ with $\{0,\ldots,n-1\}$.  When $f$ is a function and $X$ a set we write $f[X]$ for the image of $X$ under $f$.  We will also write $[X]^k$ for the set of $k$-element subsets of $X$, and $[X]^{<\omega}$ for the set of finite subsets of $X$.  We will often be concerned with the ordering of \emph{eventual dominance} on $\NN^\NN$;
\[
  f <^* g \iff \exists m\; \forall n\ge m\quad f(n) < g(n)
\]
A simple diagonalization argument shows that $\NN^\NN$ is countably directed in $<^*$.  It follows that if $X\subseteq\NN^\NN$ is cofinal in $<^*$ and is written as a countable union $X = \bigcup X_n$, then there is some $n$ for which $X_n$ is also cofinal in $<^*$.  We also note that any $X\subseteq\NN^\NN$ which is cofinal in $<^*$ must be cofinal in $<^m$ for some $m\in\NN$, where
\[
  f <^m g \iff \forall n\ge m\quad f(n) < g(n)
\]
Similar facts hold for $\SSN{P}(\NN)$ under the ordering of \emph{almost-inclusion};
\[
  \SN{A}\subseteq^* \SN{B} \iff |\SN{A}\sm\SN{B}| < \aleph_0
\]
We will often use these facts without explicit reference.

Our use of forcing axioms will be limited to two of their combinatorial consequences, \emph{Todor\v cevi\'c's Axiom} ($\TA$) and \emph{Martin's Axiom} ($\MA$).  $\TA$ and $\MA$ follow from $\PFA$, but have no large cardinal strength relative to $\ZFC$ since they can be forced together over any model of $\ZFC$ (\cite{Todorcevic.PPIT}).  $\TA$ is also well known as the \emph{Open Coloring Axiom} ($\OCA$).  Our choice of the name $\TA$ stems from the fact that other, different axioms were introduced in~\cite{ARS}, also under the name $\OCA$.  $\TA$ states;
\begin{quote}
  Let $X$ be a separable metric space, and let $[X]^2 = K_0\cup K_1$ be a partition.  Suppose $K_0$ is open, when identified with a symmetric subset of $X\times X$ minus the diagonal.  Then either
  \begin{itemize}
    \item  there is an uncountable $H\subseteq X$ such that $[H]^2\subseteq K_0$ ($H$ is \emph{$K_0$-homogeneous}), or
    \item  $X$ can be written as a countable union of sets $H_n$ ($n\in\NN$) with $[H_n]^2\subseteq K_1$ ($X$ is \emph{$\sigma$-$K_1$-homogeneous}).
  \end{itemize}
\end{quote}
$\TA$ has a remarkable influence on the set theory of the real line; for instance, it implies that the least size $\bb$ of an unbounded subset of $(\NN^\NN, <^*)$ is exactly $\aleph_2$ (\cite{Todorcevic.PPIT}).  We will occasionally make use of this fact.  $\MA$ states;
\begin{quote}
  Let $\PP$ be a poset with the countable chain condition, and let $\SSN{D}$ be a collection of $\aleph_1$-many dense subsets of $\PP$.  Then there is a filter $G\subseteq\PP$ which meets every set in $\SSN{D}$.
\end{quote}
This notation diverges from the more standard refinement, in which $\MA_\kappa$ is written for the analogous statement with $\kappa$ replacing $\aleph_1$, and $\MA$ stands for ``for all $\kappa < 2^{\aleph_0}$, $\MA_\kappa$ holds.''  However, we will be working in models of $\TA$, where $\bb = \aleph_2$; since $\MA_\kappa$ implies $\bb > \kappa$, it follows that $\MA_{\aleph_1}$ is the strongest fragment of $\MA$ which is consistent with $\TA$, hence we will shorten it to just $\MA$.  

We will make frequent use of the classical results of descriptive set theory, concerning definability properties of subsets of Polish spaces.  The interested reader may consult~\cite{Kechris.CDST} for proofs and more information.  For now we simply quote our most-used results.

\begin{fact}(Jankov-von Neumann, see~\cite[Theorem~18.1]{Kechris.CDST})
  Let $X$ and $Y$ be Polish spaces, and let $A\subseteq X\times Y$ be an analytic set.  Then there is a function $f$ with domain that of $A$, such that $f$ is measurable with respect to the $\sigma$-algebra generated by the analytic subsets of $X$.
\end{fact}

\begin{fact}(\cite[\S{29}]{Kechris.CDST})
  Let $X$ be a Polish space and $A\subseteq X$ an analytic set.  Then $A$ is measurable with respect to any complete Borel probability measure.  Moreover, $A$ has the Baire Property.
\end{fact}

\begin{fact}(\cite[Theorem~8.38]{Kechris.CDST})
  A Baire-measurable function between Polish spaces is continuous on a dense $G_\delta$.
\end{fact}

\subsection{Multipliers, topologies and lifts}
\label{subsec:topologies}

In Section~\ref{sec:intro} we defined the multiplier algebra, up to isomorphism, by a maximality property.  The multiplier algebra, like the \v Cech-Stone compactification, has many explicit realizations.  For instance, a concrete representation of $\A{M}(\A{A})$ comes with any nondegenerate representation $\rho$ of $\A{A}$ on a Hilbert space $H$, as the idealizer of $\rho[\A{A}]$ inside $\B(H)$;
\[
  \A{M}(\A{A}) \simeq \set{m\in\B(H)}{m\rho[\A{A}] + \rho[\A{A}]m\subseteq\rho[\A{A}]}
\]
It is well-known that the isomorphism type of this representation of the multiplier algebra does not depend on $\rho$.  Alternatively, one can take $\A{M}(\A{A})$ to be the idealizer of $\A{A}$ inside $\A{A}^{**}$, the enveloping von Neumann algebra of $\A{A}$.  Other, more abstract approaches go via Hilbert C*-modules and double centralizers (see~\cite{Lance.HM} and~\cite{Pedersen.CAAG}, respectively, for an excellent treatment of each).

The \emph{strict topology} on $\A{M}(\A{A})$ is that generated by the seminorms
\[
  m\mapsto \norm{ma} + \norm{am} \qquad (m\in \A{M}(\A{A}), a\in \A{A})
\]
$\A{M}(\A{A})$ is the strict completion of $\A{A}$ inside its enveloping von Neumann algebra $\A{A}^{**}$, and the strict topology coincides with the norm topology when restricted to $\A{A}$.  Hence if $\A{A}$ is separable then $\A{M}(\A{A})$ is also separable, in the strict topology.  The unit ball of $\A{M}(\A{A})$, when endowed with the strict topology, forms a Polish space.  In the case of $\A{A} = \A{K}(H)$, where $\A{M}(\A{A}) = \B(H)$, the strict topology is exactly the $\sigma$-strong-$*$ topology, and when restricted to norm-bounded subsets this coincides with both the weak and strong operator topologies.  Similarly, the strict topology on $\prod \A{A}_n$, when restricted to norm-bounded subsets, coincides with the product of the norm topologies.

We will always denote the quotient map $\A{M}(\A{A})\to\A{Q}(\A{A})$ by $\quomap$, regardless of the C*-algebra $\A{A}$.  Let $\vp$ be a $*$-homomorphism between corona algebras $\A{Q}(\A{A})$ and $\A{Q}(\A{B})$.  We say that $L$ is an \emph{$\e$-lift} of $\vp$ given that the diagram below commutes, up to a tolerance of $\e$;
\[
  \begin{tikzpicture}
    \matrix (m) [cdg.smallmatrix] {
       \A{M}(\A{A}) &    &  \A{M}(\A{B}) \\
               & \e &         \\
      \A{Q}(\A{A}) &    & \A{Q}(\A{B}) \\
    };
    \path [cdg.path]
      (m-1-1) edge node[auto]{$L$} (m-1-3)
      (m-3-1) edge node[below]{$\vp$} (m-3-3)
      (m-1-1) edge node[left]{$\quomap$} (m-3-1)
      (m-1-3) edge node[auto]{$\quomap$} (m-3-3);
  \end{tikzpicture}
\]
that is, $\norm{\pi(L(x)) - \vp(\pi(x))} \le \e\norm{\pi(x)}$ for all $x\in \A{M}(\A{A})$.  When $\e = 0$ we call $L$ simply a \emph{lift} of $\vp$.  In general we make no assumptions on the algebraic properties of $L$, or its definability; often we will work with a lift of $\vp$ given to us by a choice of representatives.  If $L$ is in fact a $*$-homomorphism we will call $L$ an \emph{algebraic lift}.  We will also often be concerned with lifts which are bicontinuous with respect to the strict topology; such maps we will call \emph{strict}.

Suppose now that $\A{A}_n$ ($n\in\NN$) is a sequence of unital C*-algebras, $\A{B}$ is a C*-algebra, and $\homo : \prod\A{A}_n \to \A{M}(\A{B})$ is a strict $*$-homomorphism taking $\bigoplus\A{A}_n$ into $\A{B}$.  Let $j_n : \A{A}_n\to \bigoplus\A{A}_m$ be the canonical embedding; then the sequence $\homo_n = \homo\circ j_n$ completely determines $\homo$.  In particular, if each $\A{A}_n$ and $\A{B}$ is separable, then we may identify $\homo$ with a member of the separable metric space
\[
  \prod \Hom(\A{A}_n, \A{B})
\]
where $\Hom(\A{A}_n,\A{B})$, the space of $*$-homomorphisms from $\A{A}_n$ to $\A{B}$, is given the point-norm topology.  It will be important to know when a given sequence in the above space determines a strict $*$-homomorphism $\prod\A{A}_n \to \A{M}(\A{B})$, i.e. when the above identification can be reversed.  For this we have the following;
\begin{prop}
  \label{hom.reverse}
  Let a sequence $(\homo_n)\in \prod \Hom(\A{A}_n,\A{B})$ be given.  Suppose that the projections $p_n = \homo_n(1_{\A{A}_n})$ ($n\in\NN$) are pairwise-orthogonal, and that their sums $e_m = \sum_{n\le m} p_n$ converge strictly in $\A{M}(\A{B})$.  Then there is a unique strict $*$-homomorphism
  \[
    \homo : \prod \A{A}_n \to \A{M}(\A{B})
  \]
  such that $\homo\circ j_n = \homo_n$ for each $n$.
\end{prop}
\begin{proof}
  Since the projections $p_n$ are pairwise-orthogonal, we may define a $*$-homomorphism $\homo : \bigoplus \A{A}_n \to \A{B}$ with $\homo\circ j_n = \homo_n$.  The projections $e_m$ are the image of an approximate unit for $\bigoplus \A{A}_n$ under $\homo$.  Thus we are in the situation of~\cite[Proposition~5.8]{Lance.HM}, and the conclusion is immediate.
\end{proof}

\subsection{Reduced products of UHF algebras}

A \emph{UHF algebra} is a C*-algebra $\A{A}$ such that for all $x_1,\ldots,x_n\in\A{A}$ and $\e > 0$, there is a C*-subalgebra of $\A{A}$, isomorphic to a full matrix algebra over $\CC$, with elements $y_1,\ldots,y_n$ satisfying
\[
  \forall i\le n\quad \norm{x_i - y_i} < \e
\]
When a UHF algebra $\A{A}$ is separable, it may be realized as a (C*-)direct limit of full matrix algebras $M_{k_n}(\CC)$.  In case $\A{A}$ is also unital, such a representing sequence may be chosen so that the connecting maps are unital.  We will be concerned exclusively with UHF algebras which are both separable and unital, and will therefore drop these two adjectives in all further discussions with hope that the result will be more readable.  We only note here that other formulations of UHF algebras, while equivalent in the separable case, are often not in the nonseparable case; the interested reader is referred to~\cite{Farah-Katsura.UHFI}.  Since we are so often concerned with the sequence of matrix algebras which makes up a UHF algebra, we set aside a term for it;
\begin{defn}
  Let $\A{A}$ be a UHF algebra.  A sequence $\A{A}_n$ ($n\in\NN$) of C*-subalgebras of $\A{A}$ is called \emph{suitable} if, for all $n\in\NN$, we have
  \begin{itemize}
    \item  $1_{\A{A}} \in \A{A}_n$, 
    \item  $\A{A}_n \subseteq \A{A}_{n+1}$, and
    \item  $\A{A}_n \simeq M_{k_n}(\CC)$ for some $k_n$.
  \end{itemize}
\end{defn}

Now let $\A{A}$ be a UHF algebra with suitable sequence $\A{A}_n \simeq M_{k_n}(\CC)$.  It follows that $k_n \mid k_{n+1}$ for each $n$.  The \emph{supernatural number} associated to $\A{A}$ is the formal product of all primes which eventually divide $k_n$, with (possibly infinite) multiplicity; e.g. when $k_n = 2^n$ the supernatural number associated to the resulting algebra is written $2^\infty$.  A classical theorem of Glimm (\cite{Glimm.UHF}) shows that the supernatural number associated to a UHF algebra is a complete invariant; moreover if $\A{A}$ and $\A{B}$ are UHF algebras with associated supernatural numbers $s$ and $t$, respectively, then $\A{A}$ embeds into $\A{B}$ if and only if $s \mid t$, i.e. every prime in the formal product $s$ appears in $t$, with multiplicity at least that of its copy in $s$.

It is well-known, and easy to prove, that UHF algebras are simple.  Hence, there is never a nonzero $*$-homomorphism $\A{A}\to\A{A}_n$ (unless the sequence $\A{A}_n$ is eventually constant, in which case $\A{A}$ is finite-dimensional).  We can, however, get close;
\begin{defn}
  \label{ce.def}
  Let $\A{A}$ be a C*-algebra and $\A{B}$ a C*-subalgebra.  A map $\theta : \A{A}\to\A{B}$ is called a \emph{conditional expectation} if the following hold;
  \begin{enumerate}
    \item\label{ce.linear}  $\theta$ is linear,
    \item\label{ce.cpc}  $\theta$ is a completely-positive contraction (see~\cite{Blackadar.OA} for a definition),
    \item\label{ce.ext}  $\theta(b) = b$ for all $b\in\A{B}$, and
    \item\label{ce.mod}  $\theta(ba) = b\theta(a)$ and $\theta(ab) = \theta(a)b$ for all $a\in\A{A}$ and $b\in\A{B}$.
  \end{enumerate}
\end{defn}

\begin{fact}
  \label{ce.fact}
  If $\A{A}$ is a UHF algebra with suitable sequence $\A{A}_n$ ($n\in\NN$), then there is a family of conditional expectations $\theta_n : \A{A}\to\A{A}_n$ satisfying $\theta_n \circ \theta_m = \theta_n$ for all $n \le m$.  In particular, we have $\theta_n(a) \to a$ for all $a\in\A{A}$.
\end{fact}
\begin{proof}[Proof sketch]
  Given $n < m$, consider $\A{A}_n'\cap\A{A}_m$.  This is a finite-dimensional, unital C*-subalgebra of $\A{A}$.  Let $\U$ be its unitary group; then $\U$ is compact, and hence has a bi-invariant Haar measure $\mu$.  Define $\theta_{n,m} : \A{A}_m\to\A{A}_m$ by
  \[
    \theta_{n,m}(a) = \int uau^* \,d\mu(u)
  \]
  where the integral above is defined weakly, i.e., entrywise for some matrix representation of $\A{A}_m$.  Then $\theta_{n,m}$ maps into $\A{A}_n$.  Moreover, $\theta_{n,m}\circ\theta_{m,p} = \theta_{n,p}$ whenever $n < m < p$.  It follows that there is a map $\theta_n : \A{A}\to\A{A}_n$ satisfying $\theta_n(a) = \theta_{n,m}(a)$ whenever $a\in\A{A}_m$.  The conditions on $\theta_n$ are easily checked.
\end{proof}

We will be concerned primarily with corona algebras of the form $\prod \A{A}_n / \bigoplus \A{A}_n$, which we call \emph{reduced products}.  The following proposition describes some of the structure of reduced products of UHF algebras, and will play an important part in the results to follow.  In particular, it leads to our definition of a coherent family of $*$-homomorphisms, which we state afterwards.
\begin{prop}
  \label{stratification}
  Let $\A{A}_n$ ($n\in\NN$) be a sequence of UHF algebras, and for each $n\in\NN$ let $\A{A}_{n,k}$ ($k\in\NN$) be a suitable sequence of subalgebras of $\A{A}_n$.  For each $\xi\in\NN^\NN$, let
  \[
    \Q_\xi = \prod \A{A}_{n,\xi(n)} / \bigoplus \A{A}_{n,\xi(n)} \subseteq \prod \A{A}_n / \bigoplus \A{A}_n
  \]
  Then $\Q_\xi \subseteq \Q_\eta$ if and only if $\xi <^* \eta$, and if $X$ is any cofinal subset of $(\NN^\NN, <^*)$, then
  \[
    \prod \A{A}_n / \bigoplus \A{A}_n = \bigcup_{\xi\in X} \Q_\xi
  \]
\end{prop}
\begin{proof}
  Let $x\in\prod \A{A}_n$ be given.  For each $n\in\NN$ we may choose some $\xi(n)\in\NN$ large enough that
  \[
    \norm{\proj{n}{x} - \A{A}_{n,\xi(n)}} \le \frac{1}{n+1}
  \]
  Hence, there is a sequence $\bar{x}\in \prod \A{A}_{n,\xi(n)}$ such that $\norm{\proj{n}{x} - \proj{n}{\bar{x}}} \to 0$ as $n\to\infty$, and so $\pi(x) = \pi(\bar{x}) \in \Q_\xi$.  The rest is straightforward.
\end{proof}

\begin{rmk}
  By Fact~\ref{ce.fact}, the $\bar{x}$ in the above proof may be chosen in a canonical way, namely we may choose $\proj{n}{\bar{x}} = \theta_{n,\xi(n)}(\proj{n}{x})$ ($n\in\NN$) where $\theta_{n,k} : \A{A}_n\to \A{A}_{n,k}$ ($k\in\NN$) is a sequence of conditional expectations fixed in advance.
\end{rmk}

\begin{defn}
  Let $\A{A}_n$ ($n\in\NN$) be a sequence of UHF algebras, and let $\A{A}_{n,k}$ ($k\in\NN$) be a suitable sequence for $\A{A}_n$, for each $n\in\NN$.  Let $\A{B}$ be a C*-algebra.  A family of $*$-homomorphisms $\homo^\xi_n : \A{A}_{n,\xi(n)} \to \A{B}$ ($\xi\in\NN^\NN$, $n\in\NN$) is called \emph{coherent} relative to the sequences $\A{A}_{n,k}$, if for each $\xi <^* \eta$, 
  \[
    \lim_n \norm{\homo^\eta_n\rs\A{A}_{n,\xi(n)} - \homo^\xi_n} = 0
  \]
\end{defn}

The following Proposition is now immediate from Propositions~\ref{hom.reverse} and~\ref{stratification}.
\begin{prop}
  \label{limit-homo}
  Suppose $\homo^\xi_n$ ($\xi\in\NN^\NN$, $n\in\NN$) is a coherent family of $*$-homomorphisms relative to suitable sequences $\A{A}_{n,k}$, all mapping into a C*-algebra $\A{B}$.  Suppose moreover that, for each $\xi$, the projections $p^\xi_n = \homo^\xi_n(1_{\A{A}_n})$ ($n\in\NN$) are pairwise-orthogonal and have sums $\sum_{n\le m} p^\xi_n$ which converge strictly in $\A{M}(\A{B})$.  Then for each $\xi$ there is a unique $*$-homomorphism $\homo^\xi : \prod \A{A}_{n,\xi(n)} \to \A{M}(\A{B})$ such that $\homo^\xi\circ j_n = \homo^\xi_n$ for every $n\in\NN$.  Moreover, if $\vp^\xi : \A{Q}_\xi \to \A{Q}(\A{B})$ is the $*$-homomorphism induced by $\homo^\xi$, then $\vp^\eta$ extends $\vp^\xi$ whenever $\xi <^* \eta$, and there is a unique $\vp : \prod \A{A}_n / \bigoplus \A{A}_n \to \A{M}(\A{B}) / \A{B}$ which extends every $\vp^\xi$.
\end{prop}
The $*$-homomorphism $\vp$ above is said to be \emph{determined} by the coherent family $\homo^\xi_n$.  We can now rephrase Theorem~\ref{main.cfh} as follows; $\TA + \MA$ implies that every isomorphism of the form $\prod \A{A}_n / \bigoplus \A{A}_n \simeq \prod \A{B}_n / \bigoplus\A{B}_n$, where each $\A{A}_n$ and $\A{B}_n$ is a UHF algebra, is determined by a coherent family of $*$-homomorphisms.  To end this section we will prove Corollary~\ref{main.cor} from Theorem~\ref{main.cfh}.  Before starting we will need one more structural result on reduced products of UHF algebras.
\begin{prop}
  \label{central.sequences}
  Let $\A{A}_n$ ($n\in\NN$) be a sequence of UHF algebras.  Then the center of $\prod \A{A}_n / \bigoplus \A{A}_n$ is canonically isomorphic to $\ell^\infty / c_0$.
\end{prop}
\begin{proof}
  Define $i : \ell^\infty \to \prod \A{A}_n$ by $\proj{n}{i(x)} = \proj{n}{x} 1_{\A{A}_n}$ for $n\in\NN$ and $x\in\ell^\infty$.  Since the center of a UHF algebra is trivial, it follows that the image of $i$ is exactly the center of $\prod\A{A}_n$.  Moreover, $i$ maps $c_0$ into $\bigoplus \A{A}_n$, and hence induces an injective map
  \[
    j : \ell^\infty / c_0 \to \Z\left(\prod\A{A}_n / \bigoplus \A{A}_n\right)
  \]
  It suffices to show that $j$ is surjective.  Suppose $x\in\prod\A{A}_n$ and $\pi(x)$ is central in $\prod\A{A}_n / \bigoplus \A{A}_n$.  Fix a suitable sequence $\A{A}_{n,k}$ ($k\in\NN$) of subalgebras for each $\A{A}_n$.  By Proposition~\ref{stratification} above, we may assume without loss of generality that for some $\xi\in\NN^\NN$, we have $\proj{n}{x}\in\A{A}_{n,\xi(n)}$ for all $n\in\NN$.  Since each $\A{A}_{n,\xi(n)}$ is finite-dimensional, its unitary group has a bi-invariant Haar measure $\mu_n$.  Then let
  \[
    \proj{n}{z} = \int u \proj{n}{x} u^* \,d\mu_n(u)
  \]
  (cf. the sketch of Fact~\ref{ce.fact} above.)  It is straightforward to show that each $\proj{n}{z}$ is scalar, and $\norm{\proj{n}{x} - \proj{n}{z}} \to 0$ as $n\to\infty$, and this completes the proof.
\end{proof}
By Proposition~\ref{central.sequences}, every isomorphism $\vp$ between reduced products of UHF algebras must restrict to an automorphism of $\ell^\infty / c_0$.  We will call this automorphism the \emph{central automorphism induced by $\vp$}.
\begin{proof}[Proof of Corollary~\ref{main.cor}]
  Let $\vp$ be an isomorphism
  \[
    \prod \A{A}_n / \bigoplus \A{A}_n \to \prod \A{B}_n / \bigoplus \A{B}_n
  \]
  and let $\sigma$ be the central automorphism induced by $\vp$.  Then by $\TA + \MA$ and the main result of~\cite{Velickovic.OCAA}, there are cofinite sets $\SN{A},\SN{B}\subseteq\NN$ and a bijection $e : \SN{A}\to\SN{B}$ such that $x\mapsto x\circ e^{-1}$ is a lift of $\sigma$.  We claim now that for all but finitely many $n\in\SN{A}$, $\A{A}_n \simeq \A{B}_{e(n)}$.  Suppose otherwise; then there is some infinite $\SN{I}\subseteq\SN{A}$ such that for all $n\in\SN{I}$, $\A{A}_n$ and $\A{B}_{e(n)}$ are not isomorphic.  Let $s_n$ and $t_n$ be the supernatural numbers associated to $\A{A}_n$ and $\A{B}_{e(n)}$, respectively; then for each $n\in\SN{I}$ there is some prime $p_n$ which divides one of $s_n,t_n$, but not the other.  Without loss of generality, $p_n\mid s_n$ and $p_n\nmid t_n$ for all $n\in\SN{I}$.  For $n\not\in\SN{I}$ let $p_n = 1$.  Now let $\A{F}_n$ be a C*-subalgebra of $\A{A}_n$ isomorphic to $M_{p_n}(\CC)$, with $1_{\A{A}_n}\in \A{F}_n$, for each $n$; by Theorem~\ref{main.cfh} there is a strict algebraic lift $\homo$ of the restriction of $\vp$ to the C*-subalgebra
  \[
    \prod \A{F}_n / \bigoplus \A{F}_n
  \]
  Let $\alpha_n : \A{F}_n\to \bigoplus \A{B}_k$ be the coordinate $*$-homomorphisms.  Notice that $\alpha_n(1_{\A{A}_n}) - 1_{\A{B}_{e(n)}}$ tends to zero as $n\to\infty$.  It follows that $\alpha_n(1_{\A{A}_n}) = 1_{\A{B}_{e(n)}}$ for all but finitely many $n\in\NN$.  Hence, for all but finitely many $n\in\NN$, there is a unital embedding of $M_{p_n}(\CC)$ into $\A{B}_{e(n)}$, i.e. $p_n \mid t_{e(n)}$.  This contradicts the previous assumption.
\end{proof}

\section{Definable embeddings}
\label{sec:definable-embeddings}

Let $\rho : \A{A}\to \A{B}$ be a map between C*-algebras.  The \emph{defect} of $\rho$ is defined to be the supremum, over all $a,b$ in the unit ball of $\A{A}$ and $t\in\CC$ with $|t|\le 1$, of the maximum of the following quantities:
\begin{gather*}
  \norm{\rho(ab) - \rho(a)\rho(b)} \\
  \norm{\rho(a+b) - (\rho(a)+\rho(b))} \\
  \norm{\rho(a^*) - \rho(a)^*} \\
  \norm{\rho(ta) - t\rho(a)} \\
  |\norm{a} - \norm{\rho(a)}|
\end{gather*}
The defect of $\rho$ thus measures how far $\rho$ is from being a $*$-homomorphism.

\begin{thm}(Farah,~\cite[Theorem 5.1]{Farah.CO})
  \label{ulam.fd}
  There is a universal constant $K_{FD}$ such that for any two finite-dimensional C*-algebras $\A{A}$ and $\A{B}$, and any Borel-measurable map $\rho : \A{A}\to \A{B}$, if the defect $\delta$ of $\rho$ is less than $1/1000$ then there is a $*$-homomorphism $\vp : \A{A}\to \A{B}$ such that $\norm{\rho - \vp} \le K_{FD}\delta$.
\end{thm}

\begin{prop}
  \label{ulam.af}
  There is a universal constant $K_{AF}$ such that for any finite-dimensional C*-algebra $\A{A}$ and AF algebra $\A{B}$, and any map $\rho : \A{A}\to \A{B}$, if the defect $\delta$ of $\rho$ is less than $10^{-6}$ then there is a $*$-homomorphism $\vp : \A{A}\to \A{B}$ such that $\norm{\rho - \vp} \le K_{AF}\delta$.
\end{prop}
\begin{proof}
  Let $\rho : \A{A}\to \A{B}$ be a map with defect $\delta$, and assume $\delta < 1/1000$.  Let $X$ be a finite, $\delta$-dense subset of the unit sphere of $\A{A}$.  Since $\A{B}$ is AF, there is a finite-dimensional C*-subalgebra $\A{C}$ of $\A{B}$ such that $\rho(x)$ is within $\delta$ of $\A{C}$ for each $x\in X$.  For each $x\in X$, fix some $c_x\in \A{C}$ within $\delta$ of $\rho(x)$, and let $\sigma : \A{A}\to \A{C}$ be the map defined by setting $\sigma(a) = \norm{a}c_x$, where $x$ is the first member of $X$ which is within $\delta$ of $a/\norm{a}$, in some fixed ordering of $X$.  It follows that $\norm{\sigma - \rho} \le 2\delta$, and hence $\sigma$ has small defect.  $\sigma$ is also, clearly, Borel-measurable.   Hence by Theorem~\ref{ulam.af} there is a $*$-homomorphism $\vp : \A{A}\to \A{C}$ close to $\sigma$, and hence close to $\rho$.
\end{proof}

The following theorem, which is the main result of this section, has at its heart an application of Proposition~\ref{ulam.af} to a sequence of functions from finite-dimensional C*-algebras into a fixed AF algebra.  The crucial detail is the independence of $K_{AF}$ from the dimension of the domain algebra.

\begin{thm}
  \label{definable->lifts}
  Assume $\TA + \MA$.  Let $\vp$ be an injective $*$-homomorphism of the form
  \[
    \prod \A{F}_n / \bigoplus \A{F}_n \to \A{M}(\A{A})/\A{A}
  \]
  where $\A{A}$ is a separable AF algebra, and each $\A{F}_n$ is a finite-dimensional C*-algebra.  Suppose $\vp$ has a lift which is strictly continuous on a dense $G_\delta$.  Then $\vp$ has a strict algebraic lift $\homo$.
\end{thm}

\begin{proof}
  Let $\e_n = 2^{-n}$ and fix an increasing approximate unit $r_n$ ($n\in\NN$) of projections in $\A{A}$.  Let $\A{F} = \prod\A{F}_n$, and for $\SN{A}\subseteq\NN$, write
  \[
    \A{F}\rs\SN{A} = \prod_{n\in\SN{A}} \A{F}_n
  \]
  and similarly $\X\rs\SN{A} = \X\cap (\A{F}\rs\SN{A})$ for subsets $\X$ of $\A{F}$.  In particular let $X_n$ be a finite, $\e_n$-dense subset of the unit ball of $\A{F}_n$, and let $\X = \prod X_n$.  Under the strict topology, $\X$ (and each $\X\rs\SN{A}$) is homeomorphic to a perfect, compact subset of the Baire space $\NN^\NN$.
  
  Now fix a lift $L : \A{F}\to\A{M}(\A{A})$ of $\vp$, which is (strictly) continuous on a dense $G_\delta$ set $G\subseteq\X$.  By a standard argument (see~\cite{Farah.CO},~\cite{Jalali-Naini},~\cite{Talagrand.Submeasure}) we may find sequences $0 = n_0 < n_1 < \cdots$ and $t_i\in\X\rs [n_i,n_{i+1})$ such that for all $x\in\X$, if $x$ extends $t_i$ for infinitely many $i$, then $x\in G$.  Now let
  \[
    t^0 = \sum t_{2i}\qquad t^1 = \sum t_{2i+1}
  \]
  (The sums converge in the strict topology.)  Also, let, $\SN{A}^0 = \bigcup [n_{2i},n_{2i+1})$ and $\SN{A}^1 = \bigcup [n_{2i+1},n_{2i+2})$.  It follows that the map
  \[
    x\mapsto L(x\rs\SN{A}^0 + t^1) + L(x\rs\SN{A}^1 + t^0) - L(t^1) - L(t^0)
  \]
  lifts $\vp$ and is continuous on all of $\X$; replacing $L$ with this map, we may assume the same holds of $L$.

  \begin{claim}
    For every $n$ and $\e > 0$ there are $k > n$ and $t\in \X\rs [n,k)$ such that for all $x,y\in\X$ extending $t$,
    \begin{enumerate}
      \item\label{stable.head}
      if for all $i < n$, $\proj{i}{x} = \proj{i}{y}$, then
      \[
        \norm{(L(x) - L(y))r_n} \le \e \quad\mbox{and}\quad \norm{r_n(L(x) - L(y))} \le \e
      \]
      \item\label{stable.tail}
      if for all $i \ge k$, $\proj{i}{x} = \proj{i}{y}$, then
      \[
        \norm{(L(x) - L(y))(1 - r_k)} \le \e\quad\mbox{and}\quad \norm{(1 - r_k)(L(x) - L(y))} \le \e
      \]
    \end{enumerate}
  \end{claim}
  \begin{proof}
    We will work towards condition~\eqref{stable.tail} first; condition~\eqref{stable.head} will then follow easily from the continuity of $L$.
    
    Fix $n$ and $\e > 0$, and for each $k > n$ define $V_k\subseteq \X\rs [n,\infty)$ by placing $x\in V_k$ if and only if there are $s,t\in\X\rs n$ with
    \[
      \norm{(L(s + x) - L(t + x))(1 - r_k)} > \e \quad\mbox{or}\quad \norm{(1 - r_k)(L(s + x) - L(t + x))} > \e
    \]
    Then, $V_k$ is an open subset of $\X\rs [n,\infty)$, by continuity of $L$.  For any given $x\in\X\rs [n,\infty)$ and $s,t\in\X\rs n$, there is some $k$ such that
    \[
      \norm{(L(s + x) - L(t + x))(1 - r_k)}, \norm{(1 - r_k)(L(s + x) - L(t + x))} \le \e
    \]
    since the difference $L(s + x) - L(t + x)$ is a member of $\A{A}$.  As $\X\rs n$ is finite, it follows that for any given $x\in\X\rs [n,\infty)$ there is some $k$ with $x\not\in V_k$.  Thus $\bigcap V_k = \emptyset$.  By the Baire Category Theorem, there must be some $m$ such that $V_m$ is not dense; then we may find $\ell\ge m$ and $s\in\X\rs [n,\ell)$ such that no $x\in\X\rs [n,\infty)$ extending $s$ can be in $V_m$.  Condition~\eqref{stable.tail} follows with the choice of $t = s$ and $k = \ell$; to complete the proof, we use continuity of $L$ to find $k \ge \ell$ and $t\in\X\rs [n,k)$ extending $s$ which satisfies~\eqref{stable.head} as well.
  \end{proof}

  We call a $t$ as in the claim an \emph{$\e$-stabilizer} for the interval $[n,k)$.  By the claim, we may construct sequences $0 = n_0 < n_1 < \cdots$ and $t_i\in\X\rs [n_i,n_{i+1})$ such that $t_i$ is an $\e_i$-stabilizer for the interval $[n_i,n_{i+1})$.  For each $\zeta < 3$ put
  \begin{align*}
    \SN{A}_\zeta & = \bigcup\set{[n_i,n_{i+1})}{i\equiv \zeta\pmod{3}} \\
    z_\zeta & = \sum \set{t_i}{i\equiv \zeta\pmod{3}}
  \end{align*}
  and define a function $L_\zeta$ by
  \[
    L_\zeta(x) = L(x + z_{\zeta+1} + z_{\zeta+2}) - L(z_{\zeta+1} + z_{\zeta+2})
  \]
  (Where $\zeta + 1$ and $\zeta + 2$ are computed mod $3$.)  Clearly, each $L_\zeta$ lifts $\vp$.  If $x\in \PA{F}_1$, let $f(x)$ be some sequence in $\X$ such that for each $n$, $\norm{\proj{n}{x} - \proj{n}{f(x)}}$ is minimal.  Then in particular, $\quo{x} = \quo{f(x)}$.  Let $q_i = r_{n_{i+2}} - r_{n_{i-1}}$, setting $n_{-1} = 0$.  Note that $q_i \perp q_j$ if $i$ and $j$ differ by at least three, and moreover
  \[
    \sum_{i\equiv \zeta\bmod{3}} q_i = 1 - r_{n_{\zeta - 1}}
  \]
  for $\zeta = 0,1,2$.  Define maps $\rho_i : \PA{F}\rs [n_i, n_{i+1}) \to q_i \A{A} q_i$ by
  \[
    \rho_i(x) = \norm{x} q_i L_\zeta(f(x / \norm{x})) q_i
  \]
  where $i\equiv \zeta\pmod{3}$.
  \begin{claim}
    The map
    \[
      x\mapsto \sum_{i\equiv \zeta\bmod{3}} \rho_i(x\rs [n_i,n_{i+1}))
    \]
    lifts $\vp$ on $\SN{A}_\zeta$.
  \end{claim}
  \begin{proof}
    Let $x\in\X\rs\SN{A}_\zeta$.  Fix an $i$ with $i\equiv\zeta\pmod{3}$, and consider $u = x\rs [n_i,\infty)$ and $v = x\rs [n_i,n_{i+1})$.  Then,
    \begin{align*}
      \lVert q_i (L_\zeta(x) & - L_\zeta(v)) \rVert = \norm{q_i (L(x + z_{\zeta+1} + z_{\zeta+2}) - L(v + z_{\zeta+1} + z_{\zeta+2}))} \\
        & \le \norm{(1 - r_{n_{i-1}})(L(x + z_{\zeta+1} + z_{\zeta+2}) - L(u + z_{\zeta+1} + z_{\zeta+2}))} \\
        & + \norm{r_{n_{i+2}}(L(u + z_{\zeta+1} + z_{\zeta+2}) - L(v + z_{\zeta+1} + z_{\zeta+2}))} \\
        & \le \e_{i-2} + \e_{i+1}
    \end{align*}
    by the properties of the stabilizers constructed above.  Similarly,
    \begin{align*}
      \norm{(L_\zeta(v) - L_\zeta(0))(1 - q_i)} & \le \norm{(L_\zeta(v) - L_\zeta(0))(1 - r_{n_{i+2}})} + \norm{(L_\zeta(v) - L_\zeta(0))r_{n_{i-1}}} \\
        & \le \e_{i+1} + \e_{i-1}
    \end{align*}
    Then, the following sums (taken over $i\equiv \zeta\pmod{3}$) converge in norm, and hence are members of $\A{A}$;
    \begin{align}
      \label{diff.rs} \sum & q_i (L_\zeta(x) - L_\zeta(x\rs [n_i,n_{i+1}))) \\
      \label{diff.pr} \sum & q_i (L_\zeta(x\rs [n_i,n_{i+1})) - L_\zeta(0))(1 - q_i) \\
      \sum & q_i L_\zeta(0)(1 - q_i)
    \end{align}
    Adding these together produces
    \[
      (1 - r_{n_{\zeta-1}})L_\zeta(x) - \sum_{i\equiv \zeta\bmod{3}} q_i L_\zeta(x\rs [n_i,n_{i+1})) q_i
    \]
    Hence the desired conclusion holds for all $x\in\X\rs\SN{A}_\zeta$.  The general case follows from the fact that $x - f(x)\in\A{A}$ whenever $x\in\PA{F}_1$.
  \end{proof}

  It follows that the defect $\delta_i$ of $\rho_i$ vanishes as $i$ tends to infinity.  For instance, if there were sequences $a_i,b_i\in \PA{F}_1\rs [n_i,n_{i+1})$ ($i\in\NN$) satisfying
  \[
    \limsup_i \norm{\rho_i(a_i b_i) - \rho_i(a_i)\rho_i(b_i)} > 0
  \]
  then letting $a = \sum a_i$ and $b = \sum b_i$, we would have $\norm{\vp(\quo{ab}) - \vp(\quo{a})\vp(\quo{b})} > 0$, a contradiction.  The other C*-algebra operations give analogous proofs.  By Proposition~\ref{ulam.af}, for large enough $i$ there is a $*$-homomorphism $\homo_i : \PA{F}\rs [n_i,n_{i+1})\to q_i \A{A} q_i$ such that $\norm{\rho_i - \homo_i} \le K_{AF}\delta_i$.  Then
  \[
    \homo^\zeta(x) = \sum\set{\homo_i(x\rs [n_i,n_{i+1}))}{i\equiv \zeta\pmod{3}}
  \]
  lifts $\vp$ on $\SN{A}_\zeta$, and is a $*$-homomorphism.  Hence $\homo = \homo^0 + \homo^1 + \homo^2$ (possibly with modifications on finitely-many coordinates) is as desired.
\end{proof}

\section{Embeddings under forcing axioms}
\label{sec:fa-embeddings}

In this section we prove the following strengthening of Theorem~\ref{main.cfh}.  The reader can easily deduce an analogous strengthening of Corollary~\ref{main.cor} using Theorem~\ref{fa->cfh.stronger} in place of Theorem~\ref{main.cfh}.
\begin{thm}
  \label{fa->cfh.stronger}
  Assume $\TA + \MA$.  Let $\A{F}_n$ and $\A{B}_n$ ($n\in\NN$) be sequences of full matrix algebras and UHF algebras, respectively, and suppose
  \[
    \vp : \prod \A{F}_n / \bigoplus \A{F}_n \to \prod \A{B}_n / \bigoplus \A{B}_n
  \]
  is an injective $*$-homomorphism which induces an automorphism of $\ell^\infty / c_0$.  Then $\vp$ has a strict algebraic lift.
\end{thm}
\begin{rmk}
  Let $\A{B}_{n,k}$ ($k\in\NN$) be suitable sequences for the UHF algebras $\A{B}_n$.  Then the conclusion of Theorem~\ref{fa->cfh.stronger} implies that there is some $\xi\in\NN^\NN$ such that the image of $\vp$ is contained in $\prod\A{B}_{n,\xi(n)} / \bigoplus \A{B}_{n,\xi(n)}$.  Indeed, let $\homo$ be a strict algebraic lift of $\vp$ and consider the coordinate $*$-homomorphisms $\homo_n = \homo\circ j_n : \A{F}_n\to \bigoplus \A{B}_m$.  By a straightforward argument, we have $\homo(1_{\A{F}_n}) - 1_{\A{B}_{e(n)}} \to 0$ for some function $e$.  Then for each $n$ we may find a $\xi(e(n))$ large enough such that there is a $*$-homomorphism $\gomo_n : \A{F}_n\to \A{B}_{e(n),\xi(e(n))}$ with $\norm{\homo_n - \gomo_n}$ tending to zero.  The sequence $\gomo_n$ then determines a strict $*$-homomorphism
  \[
    \gomo : \prod \A{F}_n \to \prod \A{B}_{n,\xi(n)}
  \]
  and $\gomo$ lifts $\vp$.
\end{rmk}
The remainder of this section is devoted to a proof of Theorem~\ref{fa->cfh.stronger}.  To this end we fix an injective $*$-homomorphism $\vp$ of the form
\[
  \prod \A{F}_n / \bigoplus \A{F}_n \to \prod \A{B}_n / \bigoplus \A{B}_n
\]
which induces an automorphism of $\ell^\infty / c_0$.  We will also assume $\TA + \MA$ for the rest of the section.  We will write $\PA{F} = \prod \A{F}_n$ and $\PA{B} = \prod\A{B}_n$, and for sets $\SN{A}\subseteq\NN$ we put
\[
  \PA{F}\rs\SN{A} = \prod_{n\in\SN{A}} \A{F}_n \qquad \PA{B}\rs\SN{A} = \prod_{n\in\SN{A}} \A{B}_n
\]
and we shorten the abominable $\vp\rs (\PA{F}\rs\SN{A})$ to just $\vp\rs\SN{A}$.  Note that, by the main result of~\cite{Velickovic.OCAA}, we may fix a function $e : \NN\to\NN$ such that for each central $\zeta\in \prod\A{F}_n$, we have $\vp(\quo{\zeta}) = \quo{\zeta\circ e}$.  By relabeling the C*-algebras $\A{B}_n$, we may, and will, assume that $e = \id$.  As in the proof of Corollary~\ref{main.cor}, this implies that any strict algebraic lift $\homo^{\SN{A}}$ of $\vp\rs\SN{A}$ must be determined by $*$-homomorphisms
\[
  \homo_n^\SN{A} : \A{F}_n\to \A{B}_n \quad (n\in\SN{A})
\]
This fact will be used often in what follows, without explicit mention.  Now, fix a pointclass $\PC{\Gamma}$ and a number $\e \ge 0$.  We define
\begin{align*}
  \SN{A}\in\SSN{I}^\e & \iff \mbox{there is a strict algebraic $\e$-lift of $\vp\rs\SN{A}$} \\
  \SN{A}\in\SSN{I}^\e(\PC{\Gamma}) & \iff \mbox{there is a $\PC{\Gamma}$-measurable $\e$-lift of $\vp\rs \SN{A}$} \\
  \SN{A}\in\SSN{I}_\sigma^\e(\PC{\Gamma}) & \iff \mbox{there is a sequence $L_k$ ($k\in\NN$) of $\PC{\Gamma}$-measurable functions with} \\
      & \qquad \forall x\in\PA{F}\rs\SN{A}\;\;\exists k\in\NN\quad \norm{\vp(\quo{x}) - \quo{L_k(x)}} \le \e\norm{\quo{x}}
\end{align*}
Our ultimate goal is to show that $\SSN{I}^0 = \SSN{P}(\NN)$.  The pointclasses we will be concerned with are $\PC{BP},\PC{H},\PC{C}$, and $\PC{\Delta^1_1}$, consisting of those sets with the Baire property, the Haar-measurable sets, the $\PC{C}$-measurable sets (see~\cite[\S{29.D}]{Kechris.CDST}), and the Borel sets, respectively.
\begin{prop}
  \label{ideals}
  Let $\e\ge 0$ and let $\PC{\Gamma}$ be a pointclass; then each $\SSN{I}^\e$, $\SSN{I}^\e(\PC{\Gamma})$, and $\SSN{I}_\sigma^\e(\PC{\Gamma})$ is an ideal containing the finite sets, and
  \[
    \SSN{I}^\e(\PC{BP}) \subseteq \SSN{I}^{8\e}(\PC{\Delta}^1_1)\qquad\mbox{and}\qquad \SSN{I}_\sigma^\e(\PC{BP}) \subseteq \SSN{I}_\sigma^{8\e}(\PC{\Delta}^1_1)
  \]
  Finally,
  \[
    \SSN{I}^0 = \SSN{I}^0(\PC{BP}) = \bigcap_{\e > 0} \SSN{I}^\e(\PC{BP})
  \]
\end{prop}
\begin{proof}
  Clearly each $\SSN{I}^\e$, $\SSN{I}^\e(\PC{\Gamma})$, and $\SSN{I}_\sigma^\e(\PC{\Gamma})$ is hereditary and contains the finite sets.  To see that e.g. $\SSN{I}_\sigma^\e(\PC{\Gamma})$ is closed under finite unions, consider $\SN{A},\SN{B}\in\SSN{I}_\sigma^\e(\PC{\Gamma})$.  Let $L_m^\SN{A}$ ($m\in\NN$) and $L_n^\SN{B}$ ($n\in\NN$) be $\PC{\Gamma}$-measurable functions which witness that $\SN{A},\SN{B}\in\SSN{I}_\sigma^\e(\PC{\Gamma})$ respectively.  Put, for all $x\in \PA{F}\rs (\SN{A}\cup\SN{B})$ and $m,n\in\NN$,
  \[
    L_{mn}^{\SN{A}\cup\SN{B}}(x) = L_m^\SN{A}(x\rs\SN{A}) + L_n^\SN{B}(x\rs(\SN{B}\sm\SN{A}))
  \]
  Then this family of functions witnesses $\SN{A}\cup\SN{B}\in\SSN{I}_\sigma^\e(\PC{\Gamma})$.

  To see that $\SSN{I}^\e(\PC{BP}) \subseteq \SSN{I}^{8\e}(\PC{\Delta}^1_1)$, let $\SN{A}\in\SSN{I}^\e(\PC{BP})$ and fix a Baire-measurable $\e$-lift $L^C$ of $\vp$ on $\SN{A}$.  Recall that the unitary group $\U$ of $\PA{F}\rs\SN{A}$ is a Polish group; hence as $L^C$ is Baire-measurable, there is a dense $G_\delta$ set $\X\subseteq\U$ on which $L^C$ is continuous.  Let
  \[
    \R = \set{(u,v)\in\U\times\U}{v\in\X\cap u^*\X}
  \]
  Then $\R$ is Borel, and has comeager sections, hence by~\cite[Theorem~8.6]{Kechris.CDST} it has a Borel-measurable uniformization $S$.  It follows that the function
  \[
    u\mapsto L^C(uS(u))L^C(S(u)^*)
  \]
  is Borel-measurable, and a $2\e$-lift of $\vp$ on $\U$.  Now, it is a standard fact that there are continuous functions $T_1,T_2,T_3,T_4 : \PA{F}\rs\SN{A}\to\U$ such that
  \[
    \sum_i T_i(x) = x
  \]
  for all $x\in\PA{F}\rs\SN{A}$.  Composing these maps with the function on $\U$ defined above, we obtain an $8\e$-lift of $\vp$ on all of $\PA{F}\rs\SN{A}$.  The inclusion $\SSN{I}_\sigma^\e(\PC{BP}) \subseteq \SSN{I}_\sigma^{8\e}(\PC{\Delta}^1_1)$ follows from similar reasoning.

  The equality $\SSN{I}^0 = \SSN{I}^0(\PC{BP})$ follows from Theorem~\ref{definable->lifts} and the fact that a Baire-measurable function is continuous on a dense $G_\delta$.  Clearly,
  \[
    \SSN{I}^0(\PC{BP}) \subseteq \bigcap_{\e > 0} \SSN{I}^\e(\PC{BP})
  \]
  Now to see the other inclusion, note by the above that
  \[
    \bigcap_{\e > 0} \SSN{I}^\e(\PC{BP}) = \bigcap_{\e > 0} \SSN{I}^\e(\PC{\Delta}^1_1)
  \]
  So, suppose that for each $\e > 0$, $\SN{A}\in\SSN{I}^\e(\PC{\Delta}^1_1)$, and let $L^\e$ be a Borel-measurable function witnessing this.  Put
  \[
    \Lambda = \set{(x,y)\in (\PA{F}_1\rs \SN{A})\times (\PA{B}_1\rs\SN{A})}{\forall \e > 0 \; \norm{\quo{L^\e(x) - y}}\le \e}
  \]
  Then, $\Lambda$ is Borel, and hence has a $\PC{C}$-measurable uniformization by the Jankov-von Neumann theorem.  This uniformization is clearly a lift of $\vp$ on $\SN{A}$.  Thus $\SN{A}\in\SSN{I}^\e(\PC{BP})$ as required.
\end{proof}

The two results below are simple modifications of \cite[Lemma~7.6]{Farah.CO} and~\cite[Proposition~7.7]{Farah.CO}, respectively; we include proofs here for completeness, but we make no claims to their originality.
\begin{lemma}
  \label{haar.expansion}
  Suppose $\vp\rs\SN{A}$ has a Borel-measurable $\e$-lift on $S$, where
  \[
    S\subseteq \prod_{n\in\SN{A}} \U(\A{F}_n) = \U
  \]
  is some set with positive Haar measure.  Then $\vp$ has a Borel-measurable $2\e$-lift on all of $\U$.
\end{lemma}
\begin{proof}
  Let $L : S\to \PA{B}\rs\SN{A}$ be a Borel-measurable $\e$-lift of $\vp$.  By Luzin's theorem, we may assume that $S$ is compact and $L$ is continuous on $S$.  Let $U$ be a basic open subset of $\U$ such that $\mu(S\cap U) > \mu(U) / 2$.  Then there are $k\in\NN$ and a finite $F$ contained in
  \[
    \prod_{n\in\SN{A}\cap k} \U(\A{F}_n)
  \]
  such that $FU = \U$.  It follows that $\mu(FS) > 1/2$.  Now define $L' : FS \to \PA{B}\rs\SN{A}$ by letting $L'(u) = L(v^* u)$ whenever $v$ is the first member of $F$ such that $v^*u \in S$.  Then $L'$ is continuous and an $\e$-lift of $\vp$ on $FS$ (noting that for each $v\in F$, $\pi(v) = \pi(1)$).  Now let
  \[
    \Lambda = \set{(u,v)\in \U\times FS}{ uv^*\in FS }
  \]
  Then the section of $\Lambda$ over a given $u\in\U$ is exactly $FS\cap u(FS)^*$, which has positive Haar measure since $\mu(FS) > 1/2$.  By \cite[Theorem~8.6]{Kechris.CDST}, it follows that $\Lambda$ has a Borel-measurable uniformization $T : \U\to \PA{B}\rs\SN{A}$.  Then the map
  \[
    u\mapsto L'(u T(u)^*) L'(T(u))
  \]
  defines a $2\e$-lift of $\vp$ on all of $\U$.
\end{proof}
\begin{prop}
  \label{sigma->one}
  Let $\e > 0$ be given.  If $\SN{A}\in \SSN{I}_\sigma^\e(\PC{H})$ is infinite and $\SN{A} = \bigcup_k \SN{A}_k$ is a partition of $\SN{A}$ into infinite sets, then there is some $k$ for which $\SN{A}_k\in\SSN{I}^{4\e}(\PC{\Delta}^1_1)$.  
\end{prop}

\begin{proof}
  Let $U_n$ be the unitary group of $\A{F}_n$.  Then
  \[
    \U_k = \prod_{n\in\SN{A}_k} U_n\quad\mbox{and}\quad \W_k = \prod_{\ell\ge k} \U_\ell
  \]
  are compact groups, and clearly
  \[
    \U \simeq \prod_k \U_k = \W_0
  \]
  We thus view each $\U_k$ and $\W_k$ as a compact subgroup of $\U$.
  
  Fix Borel functions $L_i$ $(i\in\NN)$ witnessing $\SN{A}\in\SSN{I}_\sigma^\e(\PC{H})$.  Assume, for sake of contradiction, that no $\SN{A}_k$ is a member of $\SSN{I}^{4\e}(\PC{\Delta}^1_1)$.  We will construct compact sets $\V_k\subseteq\W_k$ of positive measure (using the normalized Haar measure $\mu_k$ on $\W_k$), and elements $u_k$ of $\U_k$, such that
  \begin{enumerate}
    \item  $u_k \V_{k+1} \subseteq \V_k$, and
    \item  for all $v\in \V_{k+1}$,
    \[
      \norm{\quo{(L_k (u_0\cdots u_k v) - L(u_k))\rs \SN{A}_k}} > \e
    \]
  \end{enumerate}
  Given such sequences, note that $\V_k' = u_0\cdots u_{k-1} \V_k$ is a decreasing sequence of nonempty, compact sets in $\U$.  Thus $u_\infty = \prod_k u_k$ is a member of their intersection.  Since $u_\infty \in \PA{F}\rs\SN{A}$, there is some $k$ such that
  \[
    \norm{\quo{L_k(u_\infty) - L(u_\infty)}} \le \e
  \]
  But we have
  \[
    u_\infty = u_0\cdots u_k \prod_{\ell > k} u_\ell
  \]
  and $\prod_{\ell > k} u_\ell \in \V_{k+1}$, so by the construction,
  \[
    \norm{\quo{(L_k(u_\infty) - L(u_k))\rs \SN{A}_k}} > \e
  \]
  Since $\quo{L(u_\infty)\rs \SN{A}_k} = \quo{L(u_k)\rs\SN{A}_k}$ this provides the necessary contradiction.  Now we proceed to the construction of $u_k$ and $\V_k$.

  Suppose we are given $u_0,\ldots,u_{k-1}$ and $\V_{k-1}$ satisfying the above conditions.  Since $\V_{k-1}$ has positive Haar measure, we may find compact sets $S\subseteq \U_{k-1}$ and $T\subseteq \W_k$, each with positive measure (under their respective Haar measures), such that
  \[
    \forall x\in S\quad \mu_k\set{y\in T}{(x,y)\in\V_{k-1}} > \mu_k(T)/2
  \]
  Define $\Xi\subseteq \U_k\times\W_k\times (\PA{B}_1\rs\SN{A}_k)$ and $\Lambda\subseteq S\times \prod \U(\A{B}_n)$ by
  \begin{gather*}
    (x,y,z)\in\Xi \iff \norm{\quo{L_k(u_0\cdots u_{k-1}\cdot x\cdot y)\rs \SN{A}_k - z}} \le \e \\
    (x,z)\in \Lambda\iff \mu_k\set{y\in T}{(x,y,z)\in \Xi} > \mu_k(T)/2
  \end{gather*}
  Then $\Xi$ and $\Lambda$ are both Borel.  Suppose first that for every $x\in S$ there is some $z$ with $(x,z)\in \Lambda$.  Then by the Jankov-von Neumann uniformization theorem there is a $\PC{C}$-measurable function $f : S\to \prod \U(\A{B}_n)$ uniformizing $\Lambda$.  Since $S$ has positive measure, by Lemma~\ref{haar.expansion} $f$ cannot be a $2\e$-lift of $\vp$ on $S$, since then $\vp$ would have a (Borel-measurable) $4\e$-lift on $\SN{A}_k$, contradicting our starting assumption.  So we may find some $u_k\in S$ such that there is no $z$ with $(u_k,z)\in\Lambda$, and in particular $(u_k,L(u_k))\not\in\Lambda$.  It follows that the set
  \[
    R = \set{y\in T}{(u_k,y,L(u_k))\in \Xi\;\land\; (u_k,y)\in \V_{k-1}}
  \]
  has positive measure.  Taking $\V_k$ to be some compact subset of $R$ with positive measure finishes the construction.
\end{proof}

Recall that a family $\SSN{A}\subseteq\SSN{P}(\NN)$ is \emph{almost-disjoint} (or \emph{a.d.}) if for all distinct $\SN{A},\SN{B}\in\SSN{A}$, $\SN{A}\cap\SN{B}$ is finite.  An a.d. family $\SSN{A}$ is \emph{treelike} if there is a bijection $t : \NN\to 2^{<\omega}$ such that for all $\SN{A}\in\SSN{A}$, and all $n,m\in\SN{A}$, $t(n)\subseteq t(m)$ or $t(m) \subseteq t(n)$.  Treelike families are called \emph{neat} in~\cite{Velickovic.OCAA}.
\begin{lemma}
  \label{can.has.trees}
  If $\SSN{A}$ is a treelike, a.d. family, then $\SSN{A}\sm\SSN{I}_\sigma^\e(\PC{C})$ is countable for each $\e > 0$.
\end{lemma}
\begin{proof}
   Fix $\e > 0$, and let $X$ consist of the pairs $(\SN{A},x)$ where $\SN{A}$ is an infinite subset of some member $\tau(\SN{A})$ of $\SSN{A}$, and $x$ is in the unit ball of $\PA{F}\rs\SN{A}$.  Notice that $\tau(\SN{A})$ is unique, and $\tau$ as a map $\SSN{A}\to 2^\omega$ is continuous, since $\SSN{A}$ is treelike.  We define a coloring $[X]^2 = K_0\cup K_1$ by placing $\{(\SN{A},x),(\SN{\bar{A}},\bar{x})\}\in K_0$ if and only if
  \begin{enumerate}
    \item  $\tau(\SN{A})\neq \tau(\SN{\bar{A}})$,
    \item  for all $n\in \SN{A}\cap\SN{\bar{A}}$, $\norm{\proj{n}{x} - \proj{n}{\bar{x}}} < 1 / (n+1)$, and
    \item  there is an $n\in \SN{A}\cap\SN{\bar{A}}$ such that $\norm{\proj{n}{L(x)} - \proj{n}{L(\bar{x})}} > \e/2$.
  \end{enumerate}
  It follows that $K_0$ is open in the topology on $X$ obtained by identifying $(\SN{A},x)\in X$ with $(\SN{A},x,L(x))$, a member of the Polish space
  \[
    \pow(\NN)\times \PA{F}_1\times \PA{B}_1
  \]

  \begin{claim}
    There is no uncountable $Y\subseteq X$ such that $[Y]^2\subseteq K_0$.
  \end{claim}
  \begin{proof}
    Suppose for sake of contradiction that $Y$ is uncountable and $[Y]^2\subseteq K_0$.  Let $\SN{D} = \bigcup\set{\SN{A}}{(\SN{A},x)\in Y}$, and choose $y\in\PA{F}\rs d$ such that for all $n\in d$ there is some $(\SN{A},x)\in Y$ with $n\in \SN{A}$ and $\proj{n}{y} = \proj{n}{x}$.  Since $Y$ is $K_0$-homogeneous, it follows that for all $(\SN{A},x)\in Y$,
    \[
      \forall n\in a\quad \norm{\proj{n}{x} - \proj{n}{y}} < \frac{1}{n+1}
    \]
    In particular, $\quo{x} = \quo{y\rs a}$ for all $(\SN{A},x)\in Y$, and hence $\quo{L(x)} = \quo{L(y\rs a)}$ for all $(\SN{A},x)\in Y$.  Since $Y$ is uncountable we may find an $\bar{n}\in\NN$ such that for uncountably many $(\SN{A},x)\in Y$, we have
    \[
      \forall n\in a\sm\bar{n}\quad \norm{\proj{n}{L(x)} - \proj{n}{L(y)}} \le \e / 2
    \]
    By the separability of $B_n$ for $n < \bar{n}$, there are distinct $(\SN{A},x),(\SN{\bar{A}},\bar{x})\in Y$, both satisfying the above, such that
    \[
      \forall n < \bar{n} \quad \norm{\proj{n}{L(x)} - \proj{n}{L(\bar{x})}} \le \e
    \]
    Then $\{(\SN{A},x),(\SN{\bar{A}},\bar{x})\}\in K_1$, a contradiction.
  \end{proof}
  By $\TA$, there is a countable cover $X_p$ ($p\in\NN$) of $X$ by $K_1$-homogeneous sets.  Let $D_p$ be a countable, dense subset of $X_p$ for each $p$, and let
  \[
    \SSN{D} = \set{\tau(\SN{\bar{A}})}{p\in\NN\land (\SN{\bar{A}},\bar{x})\in D_p}
  \]
  To prove the lemma it will suffice to show that $\SSN{A}\sm\SSN{D}\subseteq\SSN{I}_\sigma^\e(\PC{C})$.
  \begin{claim}
    Let $\SN{C}\in\SSN{A}\sm\SSN{D}$.  Then there is a partition $\SN{C} = \SN{C}_0\cup \SN{C}_1$ such that for all $p\in\NN$ and all $(\SN{A},x)\in X_p$, if $\SN{A}\subseteq \SN{C}_i$ for some $i$ then for all $k$ there is $(\SN{\bar{A}},\bar{x})\in D_p$ with
    \begin{enumerate}
      \item\label{close.a}  $\SN{A}\cap k = \SN{\bar{A}}\cap k$,
      \item\label{close.x}  for all $n\in \SN{A}\cap\SN{\bar{A}}$, $\norm{\proj{n}{x} - \proj{n}{\bar{x}}} < 1 / (n+1)$.
    \end{enumerate}
  \end{claim}
  \begin{proof}
    For each $k\in\NN$, let $E_k$ be a finite subset of $X$ such that for all $p < k$ and $(\SN{A},x)\in X_p$, there is some $(\SN{\bar{A}},\bar{x})\in D_p^m\cap E_k$ satisfying~\eqref{close.a} and the following restricted form of~\eqref{close.x};
    \[
      \forall n\in \SN{A}\cap\SN{\bar{A}}\cap k\quad \norm{\proj{n}{x} - \proj{n}{\bar{x}}} < \frac{1}{n+1}
    \]
    This is possible by density of $D_p$ in $X_p$ and the fact that $\PA{F}\rs k$ is finite-dimensional (and hence has a totally-bounded unit ball).  Note that for each $(\SN{\bar{A}},\bar{x})\in E_k$, the set $\SN{C}\cap\SN{\bar{A}}$ is finite.  Let $k^+$ be minimal such that for all $(\SN{\bar{A}},\bar{x})\in E_k$, $\SN{C}\cap\SN{\bar{A}}\subseteq k^+$.  Set $k_0 = 0$ and $k_{i+1} = k_i^+$ for each $i$, and
    \[
      \SN{C}_0 = \bigcup_i \SN{C}\cap [k_{2i},k_{2i+1})\qquad \SN{C}_1 = \bigcup_i \SN{C}\cap [k_{2i+1},k_{2i+2})
    \]
    The claim follows.
  \end{proof}

  Define a set $\Lambda_p \subseteq (\PA{F}\rs \SN{C}_0) \times (\PA{B}\rs \SN{C}_0)$ by placing $(x,y)\in\Lambda_p$ if and only if for every $k\in\NN$ there is some $(\SN{\bar{A}},\bar{x})\in D_p$ such that conditions~\eqref{close.a} and~\eqref{close.x} hold (with $\SN{C}_0$ replacing $\SN{A}$), and moreover
  \[
    \forall n < k\quad \norm{\proj{n}{y} - \proj{n}{L(\bar{x})}} \le \e/2
  \]
  Clearly, $\Lambda_p$ is Borel (in the usual topology), and if $(\SN{C}_0,x)\in X_p$ then by the above claim and the $K_1$-homogeneity of $X_p$, $(x,L(x))\in \Lambda_p$.  Moreover, if $(x,y)\in \Lambda_p$, then
  \[
    \forall n\in \SN{C}_0 \norm{\proj{n}{L(x)} - \proj{n}{y}} \le \e
  \]
  Let $L_p$ be a $\PC{C}$-measurable uniformization of $\Lambda_p$ for each $p\in\NN$.  Then, since $X = \bigcup_p X_p$, the sequence $L_p$ ($p\in\NN$) is a witness to the fact that $\SN{C}_0\in\SSN{I}_\sigma^\e(\PC{C})$.  Similarly, $\SN{C}_1$ is in $\SSN{I}_\sigma^\e(\PC{C})$, and hence so is $\SN{C} = \SN{C}_0\cup \SN{C}_1$.
\end{proof}

\begin{prop}
  \label{ccc/fin}
  Every uncountable, a.d. family $\SSN{B}\subseteq\SSN{P}(\NN)$ meets $\SSN{I}^0$.
\end{prop}
\begin{proof}
  Suppose for sake of contradiction that $\SSN{B}$ is an uncountable a.d. family disjoint from $\SSN{I}^0$.  By Proposition~\ref{ideals}, we may assume that for some $\e > 0$, $\SSN{B}$ is disjoint from $\SSN{I}^\e(\PC{\Delta}^1_1)$.  By a standard application of $\MA$, we may find an uncountable, a.d. family $\SSN{A}'$ such that every $\SN{A}\in\SSN{A}'$ almost-contains infinitely many members of $\SSN{B}$.  Moreover, using $\MA$ with~\cite[Lemma~2.3]{Velickovic.OCAA}, there is an uncountable $\SSN{A} \subseteq \SSN{A}'$ and, for each $\SN{A}\in\SSN{A}$, a partition $\SN{A} = \SN{A}_0\cup \SN{A}_1$, such that for each $i < 2$ the family
  \[
    \SSN{A}_i = \set{\SN{A}_i}{\SN{a}\in\SSN{A}}
  \]
  is treelike.  By Lemma~\ref{can.has.trees}, there are uncountably many $\SN{A}\in\SSN{A}$ such that $\SN{A}_0\in\SSN{I}_\sigma^{\e/4}(\PC{C})$; and by another application of Lemma~\ref{can.has.trees}, there is then some $\SN{A}\in\SSN{A}$ such that both $\SN{A}_0$ and $\SN{A}_1$ are members of $\SSN{I}_\sigma^{\e/4}(\PC{C})$, and hence their union $\SN{A}$ is also a member of $\SSN{I}_\sigma^{\e/4}(\PC{C})$.  By Proposition~\ref{sigma->one}, since $\SN{A}$ almost-contains infinitely many members of $\SSN{B}$, there must be some $\SN{B}\in\SSN{B}\cap\SSN{I}^\e(\PC{\Delta}^1_1)$.  This contradicts our assumption.
\end{proof}

\begin{lemma}
  $\SSN{I}^0$ is a dense $P$-ideal.
\end{lemma}
\begin{proof}
  That $\SSN{I}^0$ is dense follows easily from Proposition~\ref{ccc/fin}.  To prove it's a $P$-ideal, we will first show that given any infinite sequence $\SN{A}_k$ ($k\in\NN$) of sets in $\SSN{I}^\e$, where $\e > 0$, there is some $\SN{B}\in\SSN{I}^{3\e}$ such that $\SN{A}_k\subseteq^* \SN{B}$ for all $k$.  So fix a sequence $\SN{A}_k$ ($k\in\NN$) of sets in $\SSN{I}^\e$, where $\e > 0$.  We may assume that the $\SN{A}_k$'s are pairwise disjoint.  Assume for sake of contradiction that there is no $\SN{B}\in\SSN{I}^{3\e}$ which almost-includes every $\SN{A}_k$, and for each $f : \NN\to\NN$ let
  \[
    \SN{B}_f = \bigcup\set{\SN{A}_k\cap f(k)}{k\in\NN}
  \]
  Then for every $f\in\NN^\NN$ and $k\in\NN$, $\SN{B}_f\cap \SN{A}_k$ is finite, and if $f <^* g$ then $\SN{B}_f\subseteq^* \SN{B}_g$.  We will prove that for every $f\in\NN^\NN$ there is some $g\in\NN^\NN$ such that $f <^* g$ and $\SN{B}_g\sm \SN{B}_f\not\in\SSN{I}^\e$.  By a simple recursion we may then construct a $<^*$-increasing sequence $f_\ord\in\NN^\NN$, for $\ord < \omega_1$, with $\SN{B}_{f_{\ord+1}}\sm \SN{B}_{f_\ord} \not\in\SSN{I}^\e$ for each $\ord$.  Thus the sets $\SN{B}_{f_{\ord+1}}\sm \SN{B}_{f_\ord}$ form an uncountable almost-disjoint family which is disjoint from $\SSN{I}^\e$, contradicting Proposition~\ref{ccc/fin}.

  For simplicity we will assume that $f(k) = 0$ for all $k\in\NN$, and show that for some $g\in\NN^\NN$, $\SN{B}_g\not\in\SSN{I}^\e$.  For sake of contradiction, suppose that this is not so.  Define a coloring $[\NN^\NN]^2 = K_0\cup K_1$ by
  \[
    \{g,h\}\in K_0 \iff \exists n\in \SN{B}_g\cap \SN{B}_h \quad \norm{\homo_n^{\SN{B}_g} - \homo_n^{\SN{B}_h}} > 2\e
  \]
  where for each $\SN{B}\in\SSN{I}^\e$, we have fixed a sequence of $*$-homomorphisms $\homo_n^{\SN{B}} : \A{F}_n\to\A{B}_n$ (with $\homo_n^{\SN{B}} = 0$ when $n\not\in\SN{B}$)
  according to the definition of $\SSN{I}^\e$.  It follows that $K_0$ is open when $\NN^\NN$ is given the topology obtained by identifying $g$ with $(g,\homo^{\SN{B}_g})$, a member of the Polish space
  \[
    \NN^\NN\times \prod_n \Hom(\A{F}_n, \A{B}_n)
  \]
  \begin{claim}
    There is no uncountable, $K_0$-homogeneous subset of $\NN^\NN$. 
  \end{claim}
  \begin{proof}
    Suppose $H$ is such and $|H| = \aleph_1$.  Since $\bb > \aleph_1$, there is some $\bar{h}\in\NN^\NN$ such that for every $h\in H$, $h <^* \bar{h}$.  By refining $H$ to an uncountable subset $\bar{H}$, we may assume that for some $\bar{k}\in\NN$ and some sequence of $*$-homomorphisms $\zeta_k : \A{F}_k\to \A{B}_k$ ($k < \bar{k}$), we have for all $h\in\bar{H}$ that
    \begin{enumerate}
      \item  for all $k\ge \bar{k}$, $h(k) < \bar{h}(k)$,
      \item\label{close.tail}  for all $k\ge \bar{k}$, $\norm{\homo_k^{\SN{B}_h} - \homo_k^{\SN{B}_{\bar{h}}}} \le \e$.
      \item\label{close.head}  for all $k < \bar{k}$, $\norm{\homo_k^{\SN{B}_h} - \zeta_k} \le \e$.
    \end{enumerate}
    Now, clearly, $\bar{H}$ is $K_1$-homogeneous, and this is a contradiction.
  \end{proof}
  By $\TA$, $\NN^\NN$ must be $\sigma$-$K_1$-homogeneous.  Since $\NN^\NN$ is countably directed under $<^*$, there must be some $K_1$-homogeneous set $H$ which is $<^*$-cofinal in $\NN^\NN$.  It follows that for some $\bar{k}\in\NN$, $H$ is $<^{\bar{k}}$-cofinal in $\NN^\NN$, and hence
  \[
    \SN{C} = \bigcup_{h\in H} \SN{B}_h \supseteq \bigcup_{k=\bar{k}}^\infty \SN{A}_k
  \]
  For each $n\in \SN{C}$, choose a $*$-homomorphism $\homo_n : \A{F}_n\to \A{B}_n$ from the set $\set{\homo_n^{\SN{B}_h}}{h\in H\land n\in \SN{B}_h}$.  By the $K_1$-homogeneity of $H$, then, for any $h\in H$ we have
  \[
    \forall n\in \SN{B}_h\quad \norm{\homo_n - \homo_n^{\SN{B}_h}} \le 2\e
  \]
  \begin{claim}
    There is some $\ell$ such that the sequence $\homo_n$ ($n\in \SN{C}$) forms a $3\e$-lift of $\vp$ on $\SN{A}_k$ for all $k\ge \ell$.
  \end{claim}
  \begin{proof}
    Suppose not; then there are infinitely many $k\ge\bar{k}$ such that for some $x_k\in\PA{F}_1\rs\SN{A}_k$,
    \[
      \limsup_{n\in \SN{A}_k} \norm{\homo_n(\proj{n}{x_k}) - \proj{n}{L(x_k)}} > 3\e
    \]
    For simplicity we assume that this occurs for all $k\ge \bar{k}$.  Define $x\in\PA{F}\rs\SN{A}$ by $x\rs\SN{A}_k = x_k$ for each $k\ge\bar{k}$ and $x\rs\SN{A}_k = 0$ for $k < \bar{k}$.  Then $\quo{L(x)\rs \SN{A}_k} = \quo{x_k}$ for each $k\ge\bar{k}$.  Hence for each $k\ge\bar{k}$, we may choose some $n_k\in \SN{A}_k$ large enough that
    \[
      \norm{\homo_{n_k}(\proj{n_k}{x}) - \proj{n_k}{L(x)}} > 3\e
    \]
    Define $h\in\NN^\NN$ by $h(k) = n_k + 1$ and let $y = x\rs \SN{B}_h$.  Then $\quo{L(y)} = \quo{L(x)\rs \SN{B}_h}$, and so for any $k\ge\bar{k}$ large enough,
    \[
      \norm{\homo_{n_k}(\proj{n_k}{y}) - \proj{n_k}{L(y)}} > 3\e
    \]
    But $\SN{B}_h\in \SSN{I}^\e$, and $\norm{\homo_n - \homo_n^{\SN{B}_h}} \le 2\e$ for all $n\in \SN{B}_h$.  This is a contradiction.
  \end{proof}
  \begin{claim}
    The sequence $\homo_n$ ($n\in \SN{C}$) forms a $3\e$-lift of $\vp$ on $\bigcup\set{\SN{A}_k}{k\ge\ell}$.
  \end{claim}
  \begin{proof}
    This follows from the fact that the ideal generated by $\set{\SN{B}_f}{f\in\NN^\NN}$ and $\set{\SN{A}_k}{k\ge\ell}$ is dense in $\pow(\bigcup\set{\SN{A}_k}{k\ge\ell})$.
  \end{proof}

  Since $\SN{A}_0,\ldots,\SN{A}_{\ell-1}\in\SSN{I}^\e$, and $\bigcup\set{\SN{A}_k}{k\ge\ell}\in\SSN{I}^{3\e}$, it follows that their union is in $\SSN{I}^{3\e}$.  This clearly contradicts our assumption on the sequence $\SN{A}_k$.  Now assume we've been given a sequence $\SN{A}_k$ ($k\in\NN$) from $\SSN{I}^0$.  Then for each $\ell\in\NN$ we may choose some $\SN{B}_\ell\in\SSN{I}^{1/(\ell+1)}$ such that $\SN{A}_k\subseteq^* \SN{B}_\ell$ for all $k\in\NN$.  Then we may find $\SN{C}$ such that $\SN{A}_k \subseteq^* \SN{C} \subseteq^* \SN{B}_\ell$ for all $k,\ell\in\NN$.  It follows that $\SN{C}\in\SSN{I}^0$, hence $\SSN{I}^0$ is a $P$-ideal.
\end{proof}

Finally, we are ready to prove Theorem~\ref{fa->cfh.stronger}.
\begin{proof}
  For each $\SN{A}\in\SSN{I}^0$, fix a sequence $\homo_n^{\SN{A}} : \A{F}_n\to\A{B}_n$ of $*$-homomorphisms, witnessing that $\SN{A}\in\SSN{I}^0$.  For each $\e > 0$ define a coloring $[\SSN{I}^0]^2 = K_0^\e\cup K_1^\e$ by
  \[
    \{\SN{A},\SN{B}\}\in K_0^\e\iff \exists n\in\SN{A}\cap\SN{B}\quad \norm{\homo_n^{\SN{A}} - \homo_n^{\SN{B}}} > \e
  \]
  Then $K_0^\e$ is open when $\SN{A}\in\SSN{I}^0$ is identified with $(\SN{A},\homo^{\SN{A}})$.
  \begin{claim}
    There is no uncountable, $K_0^\e$-homogeneous subset of $\SSN{I}^0$, for any $\e > 0$.
  \end{claim}
  \begin{proof}
    Suppose $H$ is $K_0^\e$-homogeneous, and has size $\aleph_1$.  Since $\SSN{I}^0$ is a P-ideal, we may form a subset $\bar{H}$ of $\SSN{I}^0$ which, under the $\subseteq^*$ ordering, is an $\omega_1$-chain dominating $H$.  By (a weakening of) $\TA$, we may assume (by going to a cofinal subset of $\bar{H}$) that $\bar{H}$ is either $K_0^{\e/2}$- or $K_1^{\e/2}$-homogeneous.  Assume the latter holds; by refining $H$ to an uncountable subset, we may assume there is some $\bar{n}$ such that for all $\SN{A}\in H$, there is some $\SN{\bar{A}}\in\bar{H}$, for which
    \[
      \SN{A}\sm\SN{\bar{A}}\subseteq \bar{n}\quad\mbox{and}\quad\forall n\in\SN{A}\sm\bar{n}\quad \norm{\homo_n^{\SN{A}} - \homo_n^{\SN{\bar{A}}}} \le \e/4
    \]
    But then any pair $\{\SN{A},\SN{B}\}\in [H]^2$ with
    \[
      \forall n < \bar{n}\quad \norm{\homo_n^{\SN{A}} - \homo_n^{\SN{B}}} \le \e
    \]
    is in $K_1^\e$, a contradiction; so $\bar{H}$ is $K_0^{\e/2}$-homogeneous.  Replacing $H$ with $\bar{H}$ and $\e$ with $\e / 2$, we may assume without loss of generality that $H$ is an increasing $\omega_1$-chain with respect to $\subseteq^*$.  Define a forcing notion $\PP$ as follows.  The conditions of $\PP$ are taken to be triples $p = (\ell_p, x_p, H_p)$, where
    \begin{enumerate}
      \item  $\ell_p\in\NN$, $x_p\in\PA{F}_1\rs \ell_p$, and $H_p\in [H]^{<\omega}$,
      \item\label{condition.spread}
      for all distinct $\SN{A},\SN{B}\in H_p$, there is some $n\in\SN{A}\cap\SN{B}\cap \ell_p$ with
      \[
        \norm{\homo_n^\SN{A}(\proj{n}{x_p}) - \homo_n^\SN{B}(\proj{n}{x_p})} > \e/2
      \]
    \end{enumerate}
    Put $p\le q$ if and only if $\ell_p\ge\ell_q$, $H_p\supseteq H_q$, and $\norm{x_p\rs \ell_q - x_q} < \e/4$.  We will argue that $\PP$ is ccc.  Suppose $\T{A}\subseteq\PP$ is uncountable.  For each $p\in \T{A}$, let $\SN{A}_p$ be the minimal member of $H_p$ with respect to $\subseteq^*$, and choose $m_p\ge\ell_p$ large enough and $\delta_p > 0$ small enough that
    \begin{itemize}
      \item  for all $\SN{A}\in H_p$, $\SN{A}_p\sm \SN{A} \subseteq m_p$ and
      \[
        \forall n\ge m_p\quad \norm{\homo_n^\SN{A} - \homo_n^{\SN{A}_p}} \le \e / 4
      \]
      \item  for all distinct $\SN{A},\SN{B}\in H_p$, there is some $n\in\SN{A}\cap\SN{B}\cap \ell_p$ with
      \[
        \norm{\homo_n^\SN{A}(\proj{n}{x_p}) - \homo_n^\SN{B}(\proj{n}{x_p})} > \e/2 + \delta_p
      \]
    \end{itemize}
    By thinning out $\T{A}$ we may assume that there are $k,\ell,m\in\NN$ and $\delta > 0$ such that for all $p\in \T{A}$, we have $|H_p| = k$, $\ell_p = \ell$, $m_p = m$ and $\delta_p \ge \delta$.  Finally, by further thinning $\T{A}$ we may assume that for all distinct $p,q\in \T{A}$,
    \begin{itemize}
      \item  $\norm{x_p - x_q} < \delta / 2$,
      \item  for all $n < m$, $\norm{\homo_n^{\SN{A}_p} - \homo_n^{\SN{A}_q}} < \e/2$, and
      \item  $H_p\cap H_q = \emptyset$.
    \end{itemize}
    Now let $p,q\in \T{A}$ be given.  Since $\{\SN{A}_p,\SN{A}_q\}\in K_0^\e$, there is some $n\in\SN{A}_p\cap\SN{A}_q$ such that $\norm{\homo_n^{\SN{A}_p} - \homo_n^{\SN{A}_q}} > \e$.  By the above it must be that $n\ge m$.  Choose $x\in\PA{F}_1\rs (n + 1)$ with $x\rs \ell = x_p$ and
    \[
      \norm{\homo_n^{\SN{A}_p}(\proj{n}{x}) - \homo_n^{\SN{A}_q}(\proj{n}{x})} > \e/2
    \]
    and put $r = (n + 1, x, H_p\cup H_q)$.  We claim that $r\in\PP$ and $r$ extends both $p$ and $q$.  The only thing to check is that $r$ satisfies~\eqref{condition.spread}; the rest is clear.  Let $\SN{A},\SN{B}\in H_p\cup H_q$ be given.  In the case where both $\SN{A}$ and $\SN{B}$ are in $H_p$, \eqref{condition.spread} holds simply because $x\rs \ell = x_p$; in the case of $\SN{A},\SN{B}\in H_q$,~\eqref{condition.spread} holds since $\norm{x_p - x_q} < \delta / 2$.  Finally, if $\SN{A}\in H_p$ and $\SN{B}\in H_q$, then since $n\ge m$,
    \[
      \norm{\homo_n^\SN{A}(\proj{n}{x}) - \homo_n^\SN{B}(\proj{n}{x})} > \norm{\homo_n^{\SN{A}_p}(\proj{n}{x}) - \homo_n^{\SN{A}_q}(\proj{n}{x})} - (\e/4 + \e/4) > \e/2
    \]
    and so~\eqref{condition.spread} is satisfied.
    
    By $\MA$, we may find an $x\in\PA{F}_1$ and an uncountable $\hat{H}\subseteq H$ such that for all distinct $\SN{A},\SN{B}\in\hat{H}$,
    \[
      \exists n\in\SN{A}\cap\SN{B} \quad \norm{\homo_n^\SN{A}(\proj{n}{x}) - \homo_n^\SN{B}(\proj{n}{x})} > \e / 2
    \]
    By our choice of $*$-homomorphisms $\homo_n^\SN{A}$, we have for all $\SN{A}\in\hat{H}$
    \[
      \limsup_{n\in\SN{A}} \norm{\homo_n^\SN{A}(\proj{n}{x}) - \proj{n}{L(x)}} = 0
    \]
    The usual pigeonhole argument shows that this is a contradiction.
  \end{proof}
  We have shown that the first alternative of $\TA$ fails for each of the partitions $[\SSN{I}^0] = K_0^\e\cup K_1^\e$; hence $\SSN{I}^0$ is $\sigma$-$K_1^\e$-homogeneous for every $\e > 0$.  Let $\e_k = 2^{-k}$ for each $k\in\NN$; then, since $\SSN{I}^0$ is a P-ideal, we may find a decreasing sequence of sets
  \[
    \SSN{I}^0 \supseteq X_0 \supseteq X_1 \supseteq \cdots
  \]
  such that each $X_k$ is $K_1^{\e_k}$-homogeneous and cofinal in $\SSN{I}^0$ in the ordering $\subseteq^*$.  By density of $\SSN{I}^0$, for each $k$ the set $\bigcup X_k$ must be cofinite; say $[m_k,\infty) \subseteq \bigcup X_k$, and $m_k < m_{k+1}$ for each $k$.  Choose any sequence of $*$-homomorphisms $\homo_n$, $n\ge m_0$, satisfying
  \[
    \homo_n \in \set{\homo_n^\SN{A}}{m_k \le n < m_{k+1}\land n\in \SN{A}\in X_k}
  \]
  It follows that, for any $\SN{A}\in\SSN{I}^0$,
  \[
    \limsup_{n\in\SN{A}} \norm{\homo_n - \homo_n^\SN{A}} = 0
  \]
  Moreover, by density of $\SSN{I}^0$, this proves that the sequence $\homo_n$ ($n\ge m_0$) makes up a lift of $\vp$ on $[m_0,\infty)$.
\end{proof}

\section{Coherent families of $*$-homomorphisms}
\label{sec:coherent-families}

\begin{thm}
  \label{cfh->borel}
  Assume $\TA$ and let $\A{A}_n$ ($n\in\NN$) be a sequence of UHF algebras.  Let $\A{B}$ be a separable C*-algebra, and suppose
  \[
    \vp : \prod\A{A}_n / \bigoplus \A{A}_n \to \A{M}(\A{B}) / \A{B}
  \]
  is determined by a coherent family of $*$-homomorphisms.  Then $\Gr{\vp}$ is Borel.
\end{thm}
\begin{proof}
  Let $\A{A}_{n,k}$ ($k\in\NN$) be a suitable sequence of subalgebras of $\A{A}_n$, and suppose
  \[
    \homo^\xi_n : \A{A}_{n,\xi(n)}\to \A{B}\qquad (\xi\in\NN^\NN, n\in\NN)
  \]
  is a coherent family of $*$-homomorphisms which determines $\vp$.  Define colorings $[\NN^\NN]^2 = K_0^\e \cup K_1^\e$, for each $\e > 0$, by placing $\{\xi,\eta\}\in K_0^\e$ if and only if
  \[
    \exists n\in\NN\;\;\exists x\in\A{A}_{n,\xi(n)}\cap\A{A}_{n,\eta(n)}\;\;\norm{\homo_n^\xi(x) - \homo_n^\eta(x)} > \e\norm{x}
  \]
  Note that $\A{A}_{n,\xi(n)}\cap \A{A}_{n,\eta(n)} = \A{A}_{n,\min(\xi(n),\eta(n))}$.  Given $\xi\in\NN^\NN$, let $\gomo^\xi$ denote the sequence
  \[
    \gomo^\xi_{n,m} = \left\{\begin{array}{ll}
      \homo_n^\xi\rs\A{A}_{n,m} & m\le\xi(n) \\
      0 & m > \xi(n)
    \end{array}\right.
  \]
  Each $K_0^\e$ is then open in the topology on $\NN^\NN$ obtained by identifying $\xi$ with
  \[
    (\xi,\gomo^\xi) \in \NN^\NN\times \prod_{n,m} \Hom(\A{A}_{n,m},\A{B})
  \]
  where we use the point-norm topology on $\Hom(\A{F}_{n,m},\A{B})$.  We claim, as usual, that there are no uncountable $K_0^\e$-homogeneous subsets of $\NN^\NN$, for any $\e > 0$.  To see this, fix some $H\subseteq\NN^\NN$ of size $\aleph_1$.  Since $\bb > \aleph_1$ we may find some $\bar{\xi}\in\NN^\NN$ such that $\xi <^* \bar{\xi}$ for all $\xi\in H$, and by refining $H$ to an uncountable subset we may assume that for some $\bar{n}$, we have for all $\xi\in H$ that
  \begin{itemize}
    \item  for all $n\ge\bar{n}$, $\xi(n) < \bar{\xi}(n)$ and
    \item  for all $n\ge\bar{n}$, $\norm{\homo_n^{\bar{\xi}}\rs\A{A}_{n,\xi(n)} - \homo_n^\xi} \le \e/2$.
  \end{itemize}
  By the uncountability of $H$ we may find distinct $\xi,\eta\in H$ with $\xi\rs\bar{n} = \eta\rs\bar{n}$ and for all $n < \bar{n}$,
  \[
    \norm{\homo_n^\xi - \homo_n^\eta} \le \e
  \]
  hence $\{\xi,\eta\}\in K_1^\e$ and so $H$ cannot be $K_0^\e$-homogeneous.

  Let $\e_k$ ($k\in\NN$) be some sequence in $\RR^+$ converging to zero.  By $\TA$ and the $\sigma$-directedness of $\NN^\NN$ under $<^*$, we may find sets
  \[
    \NN^\NN \supseteq X_0\supseteq X_1 \supseteq \cdots
  \]
  where each $X_k$ is $K_1^{\e_k}$-homogeneous, and cofinal in $\NN^\NN$ with respect to $<^*$ (cf. the argument near the end of Theorem~\ref{fa->cfh.stronger}).  Then we may construct an increasing sequence $n_k\in\NN$ ($k\in\NN$) such that each $X_k$ is $<^{n_k}$-cofinal in $\NN^\NN$.  For each $n\in [n_k,n_{k+1})$ and $i\in\NN$, choose a function $\xi_{n,i}\in X_k$ such that $\xi_{n,i}(n) \ge i$.
  
  For each $n\in\NN$ let $\theta_{n,i} : \A{A}_n\to\A{A}_{n,i}$ ($i\in\NN$) be a commuting system of conditional expectation maps, and define $\theta^\xi : \prod \A{A}_n \to \prod \A{A}_{n,\xi(n)}$ by
  \[
    \proj{n}{\theta^\xi(a)} = \theta_{n,\xi(n)}(\proj{n}{a})
  \]
  We also let $p_n = \homo^{\bar{0}}_n(1_{\A{A}_n})$ ($n\in\NN$) where $\bar{0}$ is the function with constant value $0$.  Note that, by coherence, if $\eta\in\NN^\NN$ then
  \[
    p_n - \homo^\eta_n(1_{\A{A}_n}) \to 0
  \]
  Moreover, if $p = \sum p_n$ then $\pi(p) = \pi(\homo^\eta(1))$ for all $\eta$.
  \begin{claim}
    Let $(a,b)\in\prod(\A{A}_n)_1\times \A{M}(\A{B})_1$ be given.  Then $(a,b)\in\Gr{\vp}$ if and only if $\quo{b} = \quo{\sum p_n b p_n}$ and
    \[
      \lim_n \limsup_{i\to\infty} \norm{p_n b p_n - \homo_n^{\xi_{n,i}}(\theta_{n,i}(\proj{n}{a}))} = 0
    \]
  \end{claim}
  \begin{proof}
    Suppose that $(a,b)\in\Gr{\vp}$.  Clearly $\quo{b} = \quo{pbp}$.  Find some $\eta\in \NN^\NN$ such that $\quo{a} = \quo{\theta^\eta(a)}$.  Then $\quo{b} = \quo{\alpha^\eta(\theta^\eta(a))}$, so 
    \[
      \lim_{n\to\infty} \norm{\left(\sum_{m\ge n} p_m\right) \left(pbp  - \sum_m \alpha^\eta_m(\theta_{m,\eta(m)}(\proj{m}{a}))\right)\left(\sum_{m \ge n} p_m\right)} = 0
    \]
    Since $p_m \alpha^\eta_n(x) p_k = 0$ whenever $m \neq k$, it follows that
    \[
      \lim_{n\to\infty} \sup_{m\neq k,\, m,k\ge n} \norm{ p_m b p_k} = 0
    \]
    and this implies $\pi(b) = \pi(\sum p_n b p_n)$.  Now fix $k\in\NN$; since $X_k$ is $<^*$-cofinal in $\NN^\NN$ we may choose $\eta$ as above with $\eta\in X_k$.  Then, for large enough $m\ge n_k$, we have for all $n\ge m$,
    \begin{align}
      \label{close.dom} \norm{\proj{n}{a} - \theta_{n,\eta(n)}(\proj{n}{a})} & \le \e_k \\
      \label{close.ran} \norm{p_n b p_n - \homo_n^\eta(\theta_{n,\eta(n)}(\proj{n}{a}))} & \le \e_k
    \end{align}
    Now fix $n\ge m$ and $i\ge\eta(n)$.  Then,
    \begin{equation}
      \label{close.xieta} \norm{\homo_n^{\xi_{n,i}}(\theta_{n,\eta(n)}(\proj{n}{a})) - \homo_n^\eta(\theta_{n,\eta(n)}(\proj{n}{a}))} \le \e_k
    \end{equation}
    since $\xi_{n,i}$ and $\eta$ are both members of $X_k$.  Finally, note that by~\eqref{close.dom},
    \begin{equation}
      \label{close.ieta} \norm{\theta_{n,i}(\proj{n}{a}) - \theta_{n,\eta(n)}(\proj{n}{a})} \le \e_k
    \end{equation}
    Together the inequalities~\eqref{close.ran},~\eqref{close.xieta}, and~\eqref{close.ieta} imply
    \[
      \norm{p_n b p_n - \homo_n^{\xi_{n,i}}(\theta_{n,i}(\proj{n}{a}))} \le 3\e_k
    \]
    for any $n\ge m$ and $i\ge\eta(n)$, as required.
    
    Now assume that $(*)$ holds.  Fix $k$, and choose $\eta\in X_k$ such that $\quo{a} = \quo{\theta^\eta(a)}$.  By $(*)$ and the $K_1^{\e_k}$-homogeneity of $X_k$, for all large enough $n$ and $i$ we have
    \begin{align}
      \norm{p_n b p_n - \homo_n^{\xi_{n,i}}(\theta_{n,i}(\proj{n}{a}))} & \le \e_k \\
      \norm{\theta_{n,i}(\proj{n}{a}) - \theta_{n,\eta(n)}(\proj{n}{a})} & \le \e_k \\
      \norm{\homo_n^{\xi_{n,i}}(\theta_{n,\eta(n)}(\proj{n}{a})) - \homo_n^\eta(\theta_{n,\eta(n)}(\proj{n}{a}))} & \le \e_k
    \end{align}
    Then,
    \[
      \limsup_{n\to\infty} \norm{p_n b p_n - \homo_n^\eta(\theta_{n,\eta(n)}(\proj{n}{a}))} \le 3\e_k
    \]
    Since $\homo^\eta(\theta^\eta(a))$ is a representative of $\vp(\quo{a})$, it follows that for every $k$,
    \[
      \norm{\quo{\sum p_n b p_n} - \vp(\quo{a})} \le 3\e_k
    \]
    and since $\quo{\sum p_n b p_n} = \quo{b}$, we have $(a,b)\in\Gr{\vp}$.
  \end{proof}
  The claim provides a Borel definition of $\Gr{\vp}$, hence the proof is complete.
\end{proof}

\begin{defn}
  Let $\A{A}$ be a C*-algebra, and suppose $\A{A} = \lim\A{A}_n$.  We say that $\A{A}$ has the \emph{$(\delta,\e)$-intertwining property} with respect to the sequence $\A{A}_n$ ($n\in\NN$) if, for every sequence of $*$-homomorphisms $\homo_n : \A{A}_{2n}\to \A{A}_{2n+1}$, $\gomo_n : \A{A}_{2n+1}\to\A{A}_{2n+2}$  ($n\in\NN$) satisfying, for any $n\le m$,
\[
  \begin{tikzpicture}
    \matrix (m) [cdg.smallmatrix] {
      \A{A}_{2m+1} & & \A{A}_{2m+2} \\
                  & \delta &  \\
      \A{A}_{2n}    & & \A{A}_{2n+1}    \\
    };
    \path [cdg.path]
      (m-1-3) edge node[above]{$\gomo_m$} (m-1-1)
      (m-3-1) edge node[below]{$\homo_n$} (m-3-3)
      (m-3-1) edge (m-1-1)
      (m-3-3) edge (m-1-3);
  \end{tikzpicture}
  \qquad 
  \begin{tikzpicture}
    \matrix (m) [cdg.smallmatrix] {
      \A{A}_{2m+2} & & \A{A}_{2m+3} \\
                  & \delta &  \\
      \A{A}_{2n+2}    & & \A{A}_{2n+1}    \\
    };
    \path [cdg.path]
      (m-1-1) edge node[above]{$\homo_m$} (m-1-3)
      (m-3-3) edge node[below]{$\gomo_n$} (m-3-1)
      (m-3-1) edge (m-1-1)
      (m-3-3) edge (m-1-3);
  \end{tikzpicture}
\]
there are $*$-homomorphisms $\homo,\gomo: \A{A}\to\A{A}$ such that for all $n\in\NN$,
\[
  \begin{tikzpicture}
    \matrix (m) [cdg.smallmatrix] {
      \A{A}       &        & \A{A} \\
                  &   \e   &  \\
      \A{A}_{2n}  &        & \A{A}_{2n+1}    \\
    };
    \path [cdg.path]
      (m-1-3) edge node[above]{$\gomo$} (m-1-1)
      (m-3-1) edge node[below]{$\homo_n$} (m-3-3)
      (m-3-1) edge (m-1-1)
      (m-3-3) edge (m-1-3);
  \end{tikzpicture}
  \qquad 
  \begin{tikzpicture}
    \matrix (m) [cdg.smallmatrix] {
      \A{A}       &        & \A{A} \\
                  &   \e   &  \\
      \A{A}_{2n+2}&        & \A{A}_{2n+1}    \\
    };
    \path [cdg.path]
      (m-1-1) edge node[above]{$\homo$} (m-1-3)
      (m-3-3) edge node[below]{$\gomo_n$} (m-3-1)
      (m-3-1) edge (m-1-1)
      (m-3-3) edge (m-1-3);
  \end{tikzpicture}
\]
\end{defn}

\begin{rmk}
  The diagrams above imply that for any $n < m$ we have
  \[
    \norm{\homo_m\rs\A{A}_{2n} - \homo_n},\norm{\gomo_m\rs\A{A}_{2m+1} - \gomo_n} \le 2\delta
  \]
  Moreover if $\homo,\gomo$ are as in the conclusion, then
  \[
    \norm{\homo\rs\A{A}_{2n} - \homo_n}, \norm{\gomo\rs\A{A}_{2n+1} - \gomo} \le 2\e
  \]
  and
  \[
    \norm{\homo\circ\gomo - \id}, \norm{\gomo\circ\homo - \id} \le 2\e
  \]
\end{rmk}

\begin{prop}
  Let $\A{A}_n$ ($n\in\NN$) be UHF algebras.  Then the following are equivalent.
  \begin{enumerate}
    \item\label{tfae.borel}
    Every automorphism of $\prod\A{A}_n / \bigoplus \A{A}_n$ with Borel graph has a strict, algebraic lift.
    \item\label{tfae.intertwining}
    For every $\e > 0$ there is some $\delta > 0$ such that for all large enough $n$, $\A{A}_n$ has the $(\delta,\e)$-intertwining property with respect to any suitable sequence $\A{A}_{n,k}$ ($k\in\NN$).
  \end{enumerate}
\end{prop}

\begin{proof}
  We first prove that $\lnot\eqref{tfae.intertwining}$ implies $\lnot\eqref{tfae.borel}$.  Fix a sequence $\delta_n$ ($n\in\NN$) of positive reals tending to zero.  Assuming $\lnot\eqref{tfae.intertwining}$, we may construct, by a straightforward recursion, an infinite set $\SN{I}\subseteq\NN$ and for every $n\in\SN{I}$,
  \begin{enumerate}
    \item  a suitable sequence $\A{A}_{n,k}$ ($k\in\NN$) of subalgebras of $\A{A}_n$, and
    \item  $*$-homomorphisms $\homo_{n,k} : \A{A}_{n,2k}\to \A{A}_{n,2k+1}$ and $\gomo_{n,k} : \A{A}_{n,2k+1}\to \A{A}_{n,2k+2}$ 
  \end{enumerate}
  such that for all $k \le \ell$, we have
  \[
    \begin{tikzpicture}
      \matrix (m) [cdg.smallmatrix] {
        \A{A}_{n,2\ell+2} & & \A{A}_{n,2\ell+1} \\
                    & \delta_n &  \\
        \A{A}_{n,2k}    & & \A{A}_{n,2k+1}    \\
      };
      \path [cdg.path]
        (m-1-3) edge node[above]{$\gomo_\ell$} (m-1-1)
        (m-3-1) edge node[below]{$\homo_k$} (m-3-3)
        (m-3-1) edge (m-1-1)
        (m-3-3) edge (m-1-3);
    \end{tikzpicture}
    \qquad 
    \begin{tikzpicture}
      \matrix (m) [cdg.smallmatrix] {
        \A{A}_{n,2\ell+2} & & \A{A}_{n,2\ell+3} \\
                    & \delta_n &  \\
        \A{A}_{n,2k+2}    & & \A{A}_{n,2k+1}    \\
      };
      \path [cdg.path]
        (m-1-1) edge node[above]{$\homo_\ell$} (m-1-3)
        (m-3-3) edge node[below]{$\gomo_k$} (m-3-1)
        (m-3-1) edge (m-1-1)
        (m-3-3) edge (m-1-3);
    \end{tikzpicture}
  \]
  but for any pair of $*$-homomorphisms $\homo,\gomo : \A{A}_n\to \A{A}_n$ there is some $k\in\NN$ such that either $\norm{\gomo\circ\homo_{n,k} - \id} > \e$ or $\norm{\homo\circ \gomo_{n,k} - \id} > \e$, where $\e > 0$ is fixed.  When $n\not\in\SN{I}$ we take any suitable sequence $\A{A}_{n,k}$ and let $\homo_{n,k}$ and $\gomo_{n,k}$ be the inclusion maps $\A{A}_{n,2k}\to\A{A}_{n,2k+1}$ and $\A{A}_{n,2k+1}\to\A{A}_{n,2k+2}$ respectively.  Then the families
  \[
    \homo^\xi_n = \homo_{n,\xi(n)}, \quad \gomo^\xi_n = \gomo_{n,\xi(n)} \qquad (\xi\in\NN^\NN, n\in\NN)
  \]
  are coherent, and hence determine endomorophisms $\vp$ and $\psi$ of $\prod \A{A}_n / \bigoplus \A{A}_n$ respectively.  It is easy to see that $\vp$ and $\psi$ are inverses, so $\vp$ is an automorphism.  Note that, for each $a,b\in \prod (\A{A}_n)_1$, $(a,b)\in\Gamma_\vp$ if and only if
  \begin{align*}
    \exists \xi\in\NN^\NN\;\; \exists x\in(\PA{A}_{2\xi})_1\quad \lim \norm{\proj{n}{a} - \proj{n}{x}} & = 0 \\
      \mbox{ and } \lim \norm{\homo_{n,\xi(n)}(\proj{n}{x}) - \proj{n}{b}} & = 0
  \end{align*}
  if and only if
  \begin{align*}
    \forall \xi\in\NN^\NN\;\; \forall x\in(\PA{A}_{2\xi})_1\quad \mbox{if }\lim \norm{\proj{n}{a} - \proj{n}{x}} & = 0 \\
      \mbox{ then } \lim \norm{\homo_{n,\xi(n)}(\proj{n}{x}) - \proj{n}{b}} & = 0
  \end{align*}
  so $\Gamma_\vp$ is Borel.  Now suppose $\vp$ and $\psi$ have strict algebraic lifts $\homo$ and $\gomo$; let $\homo_n$ and $\gomo_n$ ($n\in\NN$) be the coordinate $*$-homomorphisms for $\homo$ and $\gomo$, respectively.  For each $n\in I$ we may choose some $\xi(n)$ such that $\norm{\homo_n\circ\gomo_{n,\xi(n)} - \id} > \e$ or $\norm{\gomo_n\circ\homo_{n,\xi(n)} - \id} > \e$; it follows that there is some $x\in\prod\A{A}_{n,\xi(n)}$ such that $\psi(\vp(\quo{x})) \neq \quo{x}$ or $\vp(\psi(\quo{x})) \neq \quo{x}$, a contradiction.
  
  Now we show that~\eqref{tfae.intertwining} implies~\eqref{tfae.borel}.  Assume~\eqref{tfae.intertwining} and fix an automorphism $\vp$ of $\prod \A{A}_n / \bigoplus \A{A}_n$ with Borel graph.  Notice that the statement ``$\vp$ has a strict algebraic lift'' is $\PC{\Sigma}^1_2$, and condition~\eqref{tfae.intertwining} is $\PC{\Pi}^1_2$; hence both are absolute between the ground model and any forcing extension, and so we may assume $\TA$ without any loss of generality.    Theorem~\ref{definable->lifts} implies that both $\vp$ and $\vp^{-1}$ are determined by coherent families of $*$-homomorphisms, say
  \[
    \homo^\xi_n : \A{A}_{n,\xi(n)} \to \bigoplus \A{A}_n \quad \gomo^\xi_n : \A{A}_{n,\xi(n)} \to \bigoplus\A{A}_n \qquad (\xi\in\NN^\NN, n\in\NN)
  \]
  respectively.  Let $[\NN^\NN]^2 = K_0^\e\cup K_1^\e$ and $[\NN^\NN]^2 = L_0^\e \cup L_1^\e$ be the colorings defined in the proof of Theorem~\ref{cfh->borel}, given the coherent families $\homo^\xi_n$ and $\gomo^\xi_n$ respectively.  Let $M_0^\e = K_0^\e\cup L_0^\e$; then $M_0^\e$ is open in an appropriate separable metrizable topology, and $\NN^\NN$ has no uncountable, $M_0^\e$-homogeneous subsets, for any $\e > 0$.  Arguing as in Theorem~\ref{cfh->borel}, we may find $*$-homomorphisms
  \[
    \homo_{n,k} : \A{A}_{n,k} \to \bigoplus \A{A}_n \qquad \gomo_{n,k} : \A{A}_{n,k} \to \bigoplus \A{A}_n 
  \]
  such that for some sequence $\delta_n\to 0$ and all $\xi$, we have
  \[
    \norm{\homo^\xi_n - \homo_{n,\xi(n)}}, \norm{\gomo^\xi_n - \gomo_{n,\xi(n)}} \le \delta_n
  \]
  Applying the main result of~\cite{Velickovic.OCAA} to the central automorphism induced by $\vp$, we may find a function $e : \NN\to\NN$ such that $\vp(\quo{\zeta}) = \quo{\zeta\circ e}$ for all central $\zeta$.  Then for all but finitely many $n$, and all $k$, $\homo_{n,k}$ maps into $\A{A}_{e(n)}$, and $\gomo_{n,k}$ maps into $\A{A}_{e^{-1}(n)}$.   Moreover, arguing as in the proof of Corollary~\ref{main.cor}, for all but finitely many $n$ we have that $\A{A}_n\simeq \A{A}_{e(n)}$.  By composing with this isomorphism, or its inverse, we may assume that each $\homo_{n,k}$ and $\gomo_{n,k}$ maps into $\A{A}_n$.  Finally, by perturbing each $\homo_{n,k}$ and $\gomo_{n,k}$ by an amount tending to zero as $n\to\infty$, we may assume that $\homo_{n,k}$ maps into $\A{A}_{n,k'}$, and $\gomo_{n,k}$ into $\A{A}_{n,k'}$, for some large enough $k'$ depending on $k$.  We are now in a situation where we can apply condition~\eqref{tfae.intertwining}; choose $\homo_n : \A{A}_n\to\A{A}_n$ and $\gomo_n : \A{A}_n\to\A{A}_n$ such that
  \[
    \norm{\homo_n\circ \gomo_{n,k} - \id},\norm{\gomo_n\circ \homo_{n,k} - \id} \le \e_n
  \]
  It follows that the sequence $\homo_n$ ($n\in\NN$) determines a strict algebraic lift of $\vp$.
\end{proof}

\bibliography{corona}{}
\bibliographystyle{amsplain}

\end{document}